\documentclass[times,sort&compress,3p]{elsarticle} 

\makeatletter
\def\ps@pprintTitle{%
 \let\@oddhead\@empty
 \let\@evenhead\@empty
 \def\@oddfoot{}%
 \let\@evenfoot\@oddfoot}
\makeatother

\usepackage{amsfonts}
\usepackage{amsmath}
\usepackage{amsthm}
\usepackage{amssymb}
\usepackage{graphicx}
\usepackage{color}
\usepackage{natbib}
\RequirePackage[colorlinks,citecolor=blue,urlcolor=blue]{hyperref}

\usepackage{lineno,hyperref}

\newtheorem{theorem}{Theorem}[subsection]

\newtheorem{definition}{Definition}[subsection]
\newtheorem{corollary}{Corollary}[subsection]
\newtheorem{remark}{Remark}[subsection]

\newtheorem{assumptionvec}{Assumption}
\newtheorem{assumptionmat}{Assumption}
\newtheorem{assumptionten}{Assumption}

\newtheorem{lemma}{Lemma}[subsection]

\def\E{\mathrm{E}}
\def\vec{\mathrm{vec}}
\def\cov{\mathrm{cov}}
\def\var{\mathrm{var}}

\def\mbf#1{\boldsymbol{#1}}
\def\real{\mathbb R}
\def\hi#1{^{#1}}
\def\lo#1{_{#1}}

\def\tsum{\textstyle{\sum}}
\def\loo#1{\lo{(#1)}}

\def\ali{&\,}
\def\tra{\mathrm{tr}}
\def\ali{&\,}
\def\hii#1{\hi{(#1)}}

\def\inv{\hi{-1}}
\def\tprod{\textstyle{\prod}}
\def\ca#1{{\cal{#1}}}

\newcommand{\bs}{\boldsymbol}
\newcommand{\bo}{\mathbf}

\begin{document}

\begin{frontmatter}

\title{Independent component analysis for tensor-valued data\tnoteref{grant1}}

\author[address1]{Joni Virta\corref{cor}}
\cortext[cor]{Corresponding author}
\ead[url]{joni.virta@utu.fi}

\author[address2]{Bing Li}
\ead[url]{bing@stat.psu.edu}

\author[address1,address3]{Klaus Nordhausen}
\ead[url]{klaus.nordhausen@tuwien.ac.at}

\author[address1]{Hannu Oja}
\ead[url]{hannu.oja@utu.fi}

\address[address1]{Department of Mathematics and Statistics, University of Turku, 20014 Turku, Finland}
\address[address2]{Department of Statistics, Pennsylvania State University, 326 Thomas Building, University Park, Pennsylvania 16802, USA.}
\address[address3]{CSTAT - Computational Statistics, Institute of Statistics \& Mathematical Methods in Economics, Vienna University of Technology, Wiedner Hauptstr. 7, A-1040 Vienna, Austria}

\tnotetext[grant1]{The research of Joni Virta, Klaus Nordhausen and Hannu Oja was partially supported by the Academy of Finland Grant 268703. The research of Bing Li was partially supported by the National Science Foundation Grant DMS-1407537.}

\begin{abstract}
In preprocessing tensor-valued data, e.g., images and videos, a common procedure is to vectorize the observations and subject the resulting vectors to one of the many methods used for independent component analysis (ICA). However, the tensor structure of the original data is lost in the vectorization and, as a more suitable alternative, we propose the \textit{matrix-} and \textit{tensor fourth order blind identification} (MFOBI and TFOBI). In these tensorial extensions of the classic \textit{fourth order blind identification} (FOBI) we assume a Kronecker structure for the mixing and perform FOBI simultaneously on each direction of the observed tensors. We discuss the theory and assumptions behind MFOBI and TFOBI and provide two different algorithms and related estimates of the unmixing matrices along with their asymptotic properties. Finally, simulations are used to compare the method's performance with that of classical FOBI for vectorized data and we end with a real data clustering example.
\end{abstract}


\end{frontmatter}


\section{Introduction}

\subsection{Review of matrix-valued data with the Kronecker structure}

In this paper we develop the theory and algorithms for \textit{independent component analysis} (ICA) for tensor-valued data. As the main ideas are best illustrated in the special case where the observations are matrix-valued we begin by considering the following location-scatter model incorporating Kronecker structure for matrix-valued random elements:
\begin{align}\label{eq:ls_model}
\textbf{X} = \boldsymbol{\mu} + \boldsymbol{\Omega}_L \textbf{Z} \boldsymbol{\Omega}{}_R^\top,
\end{align}
where $\textbf{X}\in \mathbb{R}^{p \times q} $ is the observed matrix, $\boldsymbol{\mu}\in \mathbb{R}^{p \times q} $ is a location center, $\boldsymbol{\Omega}_L \in \mathbb{R}^{p \times p}$ and $\boldsymbol{\Omega}_R \in \mathbb{R}^{q \times q}$ are \textit{mixing matrices} that specify linear row and column dependencies, respectively, and $\textbf{Z} \in \mathbb{R}^{p \times q}$ is a matrix of standardized uncorrelated random variables, $\E \left\{ \vec (\textbf{Z}) \right\}= \textbf{0}_{pq}$ and $\cov \left\{ \vec (\textbf{Z}) \right\} = \textbf{I}_{pq}$.

It follows that $\E \left\{ {\rm vec} (\textbf{X}) \right\}= \vec (\boldsymbol{\mu})$ and the covariance matrix of the vectorized observation has the Kronecker covariance structure,
\begin{align*}
\cov \left\{ \vec (\textbf{X}) \right\} = \boldsymbol{\Sigma}_R \otimes \boldsymbol{\Sigma}_L,
\end{align*}
with $\boldsymbol{\Sigma}_R = \boldsymbol{\Omega}_R \boldsymbol{\Omega}{}_R^\top$ and $\boldsymbol{\Sigma}_L = \boldsymbol{\Omega}_L \boldsymbol{\Omega}{}_L^\top$. Note that the structured $\cov \left\{ \vec (\textbf{X}) \right\}$ has $(1/2) p(p+1) + (1/2) q(q+1)-1$ parameters while the number of parameters in the general unstructured case is as large as $(1/2) pq(pq+1)$.

Many examples of matrix-valued data with Kronecker structure exist. For example, in the case of clustered multivariate data the i.i.d. observations $\textbf{X}_1, \ldots ,\textbf{X}_n$ represent the $n$ clusters with $q$ individuals in each cluster and $p$ variables measured on each individual, whereas in repeated measures analysis one considers $n$ individuals $\textbf{X}_1, \ldots ,\textbf{X}_n$ with $p$ measured variables and $q$ repetitions on each individual. If the columns of $\textbf{X}$ are exchangeable random vectors, as is the case with clustered data, then $\boldsymbol{\Sigma}_R$ has the intraclass correlation structure, $\boldsymbol{\Sigma}_R \propto (1 - \rho)\textbf{I}_q+ \rho \textbf{1}_q \textbf{1}_q^\top$. In applications of matrix or tensor-valued data such as  channel modelling for multiple-input multiple-output (MIMO) communication, analysis of spatio-temporal EEG (electroencephalography) data, fMRI (functional Magnetic Resonance Imaging) data, or general image or video clip data, for example, the problem itself often suggests Kronecker structure \citep{werner2008estimation}.

Consider next applying distributional assumptions for $\textbf{Z}$ in the model \eqref{eq:ls_model}. The (parametric) {\it multivariate normal model} or the wider (semiparametric) {\it elliptical model} are obtained if one assumes that $\vec (\textbf{Z})\sim \mathcal{N}_{pq}(\textbf{0}_{pq}, \textbf{I}_{pq})$ or that the distribution of $\vec (\textbf{Z})$ is spherically symmetric, respectively. In these models  $\boldsymbol{\Omega}_L$ and $\boldsymbol{\Omega}_R$ are well-defined only up to postmultiplication by orthogonal matrices and the number of free mixing parameters is therefore $(1/2) p(p+1) + (1/2) q(q+1)-1$. See for example \cite{GuptaNagar2000} for an overview of matrix-valued distributions. In this paper we assume that the $pq$ components of $\vec (\textbf{Z})$ are mutually independent. This semiparametric model, called the {\it independent component model}, provides an alternative extension of the multivariate normal model. In this case $\boldsymbol{\Omega}_L$ and $\boldsymbol{\Omega}_R$ are well-defined up to permutations and signs of their columns making the number of free mixing parameters $p^2+q^2-1$. In independent component analysis for matrix-valued data the objective is then to use the realisations $\textbf{X}_1, \ldots ,\textbf{X}_n$ of the model \eqref{eq:ls_model} to estimate \textit{unmixing matrices} $\boldsymbol{\Gamma}_L \in \mathbb{R}^{p \times p}$ and $\boldsymbol{\Gamma}_R \in \mathbb{R}^{q \times q}$ such that $\boldsymbol{\Gamma}_L \textbf{X} \boldsymbol{\Gamma}{}_R^\top$ has mutually independent components.

In the multivariate normal case \cite{srivastava2008models} introduced likelihood ratio test for the null hypotheses of Kronecker covariance structure and used the so-called flip-flop algorithm to find maximum likelihood estimates of $\boldsymbol{\Sigma}_R$ and $\boldsymbol{\Sigma}_L$ under the null hypothesis. For another approach to this estimation problem, see \cite{Ros2016, wiesel2012geodesic}. \cite{srivastava2008models} also tested the hypothesis that $\boldsymbol{\Sigma}_R$ is an identity matrix, a diagonal matrix or of intraclass correlation structure, see their paper for further references. \cite{sun2015robust} considered robust estimation of a structured covariance matrix, including Kronecker covariance structure, under heavy-tailed elliptical distributions and \cite{greenewald2014robust} modelled the covariance matrix of spatio-temporal data as a sum of low-rank Kronecker products and a sparse matrix.

\subsection{Review of methods for general tensor-valued data}

Like matrices, also tensor-valued observations have become a prevalent form of modern data and some fields of application include, e.g., psychometrics, chemometrics and computer vision, see \cite{kolda2009tensor, lu2011survey} along with the references therein for more examples.
For modelling tensor data, e.g., tensor normal distribution has been proposed, see \cite{manceur2013maximum, ohlson2013multilinear}. Also a general location-scatter model and an independent component model for tensor-valued data are easily defined, see Section~\ref{section:tensor}. In both cases, for a tensor-valued random element $\textbf{X}\in \mathbb{R}^{p_1 \times  \ldots  \times p_r}$, the covariance matrix of the vectorized observation again exhibits a Kronecker structure, $\cov \left\{ \vec (\textbf{X}) \right\} = \boldsymbol{\Sigma}_r \otimes \ldots \otimes \boldsymbol{\Sigma}_1$.

Tensor-based methods have a long history in, e.g., signal processing in the form of different tensor decompositions. The two most prevalent ones are CP-decomposition and the Tucker decomposition which provide tensor analogies for singular value decomposition and principal component analysis, respectively. Both are thoroughly discussed in \cite{kolda2009tensor} where a review of numerous other tensor decompositions is also given. See also \cite{beckmann2005tensorial} who introduce tensor PICA, an independent component analysis method for fMRI data that is based on the CP-decomposition and \cite{kim2013robust} who present various robust and sparse tensor decompositions for coping with outliers and sparsity.

Also in the statistics literature methods for tensor-valued observations have been increasingly discussed in the recent years. For example, \cite{li2010dimension} expanded the sliced inverse regression methodology developed in \cite{li1991sliced} to create dimension folding, a supervised dimension reduction method for matrix and tensor-valued predictors. \cite{pfeiffer2012sufficient} considered sufficient dimension reduction for longitudinal predictors. \cite{hung-2013, zhou-2013} developed logistic regression and generalized linear models for tensor-valued predictors. \cite{zhao-2014, zhou-2014} developed regularized linear regression and generalized linear models for tensor-valued predictors. \cite{ding2014dimension} discussed matrix versions of principal component analysis and principal fitted components (PFC). \cite{xue2014sufficient} introduced central mean dimension folding subspace and proposed several methods to estimate it. \cite{ding2015tensor} further developed tensor-valued sliced inverse regression.  An alternative   perspective for sufficient dimension reduction for tensors was considered in \cite{ding2015higher, zhong-2015}. See also \cite{hung-2012, schott-2014, zeng-2013}.

High dimensionality is common to modern, naturally tensor-valued data sets and in many cases the number of variables further exceeds the number of observations, preventing the use of vector-valued methods. In such cases tensorial methods of dimension reduction, such as those listed above, provide an especially attractive course of action, allowing the reduction of the data while taking into account its special tensor structure, see \cite{virta2016applying, virta2017blind}. In this paper we tackle this problem from the viewpoint of independent component analysis.

\subsection{Independent component analysis for tensor-valued data}

Extending independent component analysis to tensors has also seen some attention but, to our knowledge, no model-based treatise has been given. \cite{vasilescu2005multilinear, zhang2008directional} discuss the ICA problem for tensor data and propose unmixing each of the modes separately by $m$-flattening the data tensor and subjecting the matrix of $m$-mode vectors to standard ICA methods. This approach however discards all the information on the structural dependence present in the tensors. 
Our proposed method, TFOBI, a tensor analogy for a popular independent component analysis method called fourth order blind identification (FOBI) \citep{cardoso1989source}, also considers each mode separately, but instead
takes advantage of this structural information in estimating
the independent components. 

In the classic independent component analysis for vector-valued data it is assumed that the observations  $\textbf{x} \in \mathbb{R}^p$ obey the model
\begin{align} \label{eq:vec_icmodel}
\textbf{x} = \boldsymbol{\mu} + \boldsymbol{\Omega} \textbf{z},
\end{align}
where $\boldsymbol{\mu} \in \mathbb{R}^p$ is the location center, $\boldsymbol{\Omega} \in \mathbb{R}^{p \times p}$ is the so-called mixing matrix and $\textbf{z} \in \mathbb{R}^p$ is a vector of standardized, mutually independent components.
The goal is, given the i.i.d. observations $\textbf{x}_1, \ldots ,\textbf{x}_n$, to find an  estimate of an unmixing matrix  $\boldsymbol{\Gamma} \in \mathbb{R}^{p \times p}$  such that $\boldsymbol{\Gamma} \textbf{x}$ has mutually independent components. Numerous methods for solving the vector-valued independent component problem can be found in the literature, the most popular ones including FOBI, JADE (joint approximate diagonalization of eigen-matrices) and FastICA, see, e.g., \cite{HyvarinenKarhunenOja:2001, miettinen2014fourth}.

FOBI is based on the fact that in the independent component model \eqref{eq:vec_icmodel} both
\[
\E \left( \textbf{z} \textbf{z}^\top \right) = \textbf{I}_p \quad \mbox{and} \quad \E \left( \textbf{z} \textbf{z}^\top \textbf{z} \textbf{z}^\top \right) = \E \left( \| \textbf{z} \|^2 \textbf{z} \textbf{z}^\top \right)
\]
are diagonal matrices. In a similar way our extension of FOBI for matrix-valued observations, called \textit{matrix fourth order blind identification} (MFOBI), makes use of the fact that the matrices
\[
\E \left( \textbf{Z} \textbf{Z}^\top \right) = q \textbf{I}_p \quad \mbox{and} \quad \E \left( \textbf{Z}^\top \textbf{Z} \right) = p \textbf{I}_q
\]
and
\[
\E \left( \textbf{Z} \textbf{Z}^\top \textbf{Z} \textbf{Z}^\top \right) \quad \mbox{and} \quad \E \left( \textbf{Z}^\top \textbf{Z} \textbf{Z}^\top \textbf{Z}\right)
\]
and
\[
\E \left( \|\textbf{Z}\|_F^2 \textbf{Z} \textbf{Z}^\top \right) \quad \mbox{and} \quad \E \left( \|\textbf{Z}\|_F^2 \textbf{Z}^\top \textbf{Z} \right)
\]
are all diagonal. Here $\|\cdot\|_F$ is the Frobenius norm. Similar constructs for tensor-valued data are discussed in Section \ref{section:tensor}.

This paper is structured as follows. We start with some notation and important concepts in Section \ref{sec:nota}.
In Section \ref{sec:icm} we review the classic independent component model for vector-valued observations and then extend the model for matrix-valued data. The identifiability constraints and assumptions regarding both models are also discussed. Next, in Section \ref{sec:fold}, we first review the basic steps --- standardization and rotation --- of finding the classical FOBI solution and then by analogy find the MFOBI solution by double standardization and double rotation. Furthermore, we provide two different ways for estimating the double rotation and then show that the MFOBI estimate is Fisher consistent. In Section \ref{section:tensor} we further extend the method to tensor-valued data and obtain the general TFOBI method. In Section \ref{sec:asymp} we provide the asymptotic behavior for the extended FOBI procedures in the case of identity mixing. Orthogonal equivariance of TFOBI implies that the asymptotic variances derived for both versions allow comparisons with FOBI also for any orthogonal mixing matrices. In Section \ref{sec:simu} we use simulations to compare TFOBI with vectorizing and using FOBI in both the general case of estimating the correct unmixing matrix and blind classification. Also a real data example is included. Finally, in Section \ref{sec:conc} we close with some conclusions and prospective ideas.

Our route of exposition from MFOBI to TFOBI is not the most parsimonious one as MFOBI is logically a special case of TFOBI. We choose this path not only because the core ideas are best explained in the matrix setting; they would be hard to discern amongst the complicated tensor manipulations, but also because the asymptotic behavior of TFOBI reverts to that of MFOBI for tensors of all orders.

\section{Notation} \label{sec:nota}

\subsection{Some moments and cross-moments}

\def\bull{\mbox{\LARGE{\raisebox{-3pt}{$\cdot$}}}}

Next, we define some particular moments and expressions based on the moments of the elements of the i.i.d. random vectors $\textbf{z}_i$ from the distribution of $\textbf{z} \in \mathbb{R}^{p}$ and i.i.d. random matrices $\textbf{Z}_i$ from the distribution of $\textbf{Z} \in \mathbb{R}^{p \times q}$. The components of $\textbf{z}$ and $\textbf{Z}$ are mutually independent and standardized to have zero means and unit variances. Beginning with the marginal moments of the vectors we write
\[ \gamma_k := \E (z_{k}^3), \quad \beta_k := \E (z_{k}^4), \quad \text{and} \quad \omega_k := \var (z_{k}^3), \quad \forall k=1, \ldots ,p. \]
For the matrix version we require the same moments and thus define
\[ \gamma_{kl} := \E (z_{kl}^3), \quad \beta_{kl} := \E (z_{kl}^4), \quad \text{and} \quad \omega_{kl} := \var (z_{kl}^3), \quad \forall k=1, \ldots ,p, \, \forall l=1, \ldots ,q. \]
Interestingly, MFOBI involves the row and column means of the previously defined moments and we use the notation $\bar{\alpha}_{k\bull}$ to denote taking the average over the values of the bulleted index, e.g., $\bar{\omega}_{k\bull} = (1/q)\sum_{l} \omega_{kl}$. Additionally, we are going to need the covariance of two rows of kurtoses and define $\delta_{kk'} = (1/q)\sum_{l} \beta_{kl} \beta_{k'l} - \bar{\beta}_{k\bull} \bar{\beta}_{k'\bull}$.

For the asymptotic behavior of FOBI we require the following cross-moment estimates for distinct $k, k', m = 1, \ldots ,p$:
\begin{align*}
\hat{s}_{kk'} := \frac{1}{n}\sum_{i=1}^n z_{i,k} z_{i,k'}, \quad \hat{q}_{kk'} := \frac{1}{n}\sum_{i=1}^n (z_{i,k}^3 - \gamma_k) z_{i,k'} \quad \text{and} \quad \hat{q}_{mkk'} := \frac{1}{n}\sum_{i=1}^n z_{i,m}^2 z_{i,k} z_{i,k'}.
\end{align*}

For their matrix counterparts, we need both the ``left'' and ``right'' versions, e.g.,
\begin{align*}
\bar{s}^L_{kk'} := \frac{1}{q} \sum_{l=1} \hi  q \left( \frac{1}{n} \sum_{i=1}^n z_{i,kl} z_{i,k'l} \right), \quad \bar{s}^R_{ll'} := \frac{1}{p} \sum_{k=1}\hi p \left( \frac{1}{n} \sum_{i=1}^n z_{i,kl} z_{i,kl'} \right),
\end{align*}
where a bar ($\bar{a}$ instead of $\hat{a}$) is used to emphasize the taking of the mean and to avoid confusion with $\hat{s}_{kk'}$. Notice also how $\bar{s}^L_{kk'}$ and $\bar{s}^R_{ll'}$ are again the row and column averages of the corresponding vector quantities. We also see that the right-hand side version of the quantity is obtained from the left-hand side version by simply reversing the roles of rows and columns (or transposing the matrices $\textbf{Z}_i$). Due to this connection we next state only the left-hand side versions of the remaining needed quantities, also omitting the superscript ``L'':
\begin{align*}
\bar{q}_{kk'} :=  \frac{1}{q} \sum_{l=1}\hi q \left\{ \frac{1}{n}\sum_{i=1}^n (z_{i,kl}^3 - \gamma_{kl}) z_{i,k'l} \right\}, \
 \bar{q}_{mkk'} := \frac{1}{q} \sum_{l=1} \hi q \left\{ \frac{1}{n}\sum_{i=1}^n z_{i,ml}^2 z_{i,kl}  z_{i,k'l} \right\},
\end{align*}
and the following which lack a vector counterpart:
\begin{align*}
&\bar{r}_{kk'} := \frac{1}{q} \sum_{l = 1}^q \sum_{l' = 1, \, l' \neq l}^q \left( \frac{1}{n}\sum_{i=1}^n z_{i,kl}^2 z_{i,kl'} z_{i,k'l'} \right), \quad \bar{r}^0_{mkk'} := \frac{1}{q} \sum_{l = 1}^q \sum_{l' = 1, \, l' \neq l}^q \left( \frac{1}{n}\sum_{i=1}^n z_{i,kl} z_{i,ml} z_{i,ml'} z_{i,k'l'} \right) \quad \text{and} \\
& \hspace{3.5cm} \bar{r}^1_{mkk'} := \frac{1}{q} \sum_{l = 1}^q \sum_{l' = 1, \, l' \neq l}^q \left( \frac{1}{n}\sum_{i=1}^n z_{i,ml}^2 z_{i,kl'} z_{i,k'l'} \right).
\end{align*}
Assuming that the eighth moments of $\textbf{Z}$ exist the joint limiting distribution of the above quantities can be shown to be multivariate normal. Additional properties of the quantities are discussed in the proof of Theorem \ref{theo:micm_asymp} in Section \ref{sec:asymp}. Furthermore, similar quantities could also be defined for random tensors, but they are not needed in the exposition as it is later shown that the asymptotical behavior of TFOBI reduces to that of MFOBI.

\subsection{Notations for matrices and sets of matrices}



An inverse square root $\textbf{S}^{-1/2}$ of a symmetric, positive definite matrix $\textbf{S} \in \mathbb{R}^{p \times p}$ is any matrix $\textbf{G} \in \mathbb{R}^{p \times p}$ satisfying $\textbf{G} \textbf{S} \textbf{G}^\top = \textbf{I}_p$. Given the eigendecomposition of the matrix $\textbf{S} = \textbf{U} \textbf{D} \textbf{U}^\top$, all possible inverse square root
matrices of $\textbf{S}$ are of the form $\textbf{V} \textbf{D}^{-1/2} \textbf{U}^\top$, where $\textbf{V} \in \mathbb{R}^{p \times p}$ is an orthogonal matrix. If $\textbf{S}$ has distinct eigenvalues, then a unique, symmetric choice for $\textbf{S}^{-1/2}$ is $\textbf{U} \textbf{D}^{-1/2} \textbf{U}^\top$, see, e.g., \cite{ilmonen2012invariant}.

The $p$-vector $\textbf{e}_k$, $k=1, \ldots ,p$, is a vector with $k$th element one and other elements zero and ${\textbf{E}}^{kl} := \textbf{e}_k \textbf{e}_l^\top$ is a $p\times p$ matrix with $(k,l)$-element one and other elements zero, $k,l=1, \ldots ,p$. Note that $\textbf{I}_{p}=\sum_{k=1}^p \textbf{E}^{kk}$ and all diagonal matrices with diagonal elements  $c_1, \ldots ,c_p$
can be written as $\sum_{k=1}^p c_k \textbf{E}^{kk}$.


Table \ref{tab:nota_sets} lists some particular sets of (affine transformation) matrices used in the following sections. A permutation matrix is obtained if we permute the rows and/or columns of an identity matrix. A heterogeneous sign-change matrix is a diagonal matrix with diagonal entries $\pm 1$. A heterogeneous scaling matrix is a diagonal matrix with positive diagonal entries.


\begin{table}[h]
\centering
\caption{Some useful sets of square matrices}
\begin{tabular}{l|l}
Set & Description \\
\hline
$\mathcal{A}^r$ & The set of all $r \times r$ non-singular matrices. \\
$\mathcal{U}^r$ & The set of all $r \times r$ orthogonal matrices. \\
\hline
$\mathcal{P}^r$ & The set of all $r \times r$ permutation matrices. \\
$\mathcal{J}^r$ & The set of all $r \times r$ heterogeneous sign-change matrices. \\
$\mathcal{D}^r$ & The set of all $r \times r$ heterogeneous scaling matrices. \\
$\mathcal{C}^r$ & The set of all matrices $\textbf{PJD}$, where $\textbf{P} \in \mathcal{P}^r$, $\textbf{J} \in \mathcal{J}^r$ and $\textbf{D} \in \mathcal{D}^r$. \\
\end{tabular}

\label{tab:nota_sets}
\end{table}




\section{Independent component models}\label{sec:icm}

In this section we derive the basic model behind MFOBI by expanding the classic \textit{independent component model} from vector-valued to matrix-valued observations. 

\subsection{Vector-valued independent component model}

\begin{definition}\label{def:vicm}
The vector-valued independent component model assumes that the observed i.i.d. variables $\textbf{x}_i \in \mathbb{R}^p$, $i=1, \ldots ,n$, are realisations of a random vector $\textbf{x}$ satisfying
\[\textbf{x} = \boldsymbol{\mu} + \boldsymbol{\Omega} \textbf{z}, \]
where $\boldsymbol{\mu} \in \mathbb{R}^p, \boldsymbol{\Omega} \in \mathcal{A}^p$ and the random vector $\textbf{z} \in \mathbb{R}^p$ satisfies Assumptions \ref{assu:vicm_const} and \ref{assu:vicm_gauss} below.
\end{definition}

\begin{assumptionvec}\label{assu:vicm_const}
The components $z_k$ of $\textbf{z}$ are mutually independent and standardized in the sense that $\E (z_k) = 0$ and $\var (z_k) = 1$.
\end{assumptionvec}

\begin{assumptionvec}\label{assu:vicm_gauss}
At most one of the components $z_k$ of $\textbf{z}$ is normally distributed.
\end{assumptionvec}

Without Assumption \ref{assu:vicm_const} the model itself in Definition \ref{def:vicm} is not well-defined in the sense that replacing $\boldsymbol{\Omega}$ and $\textbf{z}$ with $\boldsymbol{\Omega}^* = \boldsymbol{\Omega} \textbf{C}$ and $\textbf{z}^* = \textbf{C}^{-1} \textbf{z}$, for some $\textbf{C} \in \mathcal{C}^p$, yields exactly the same model for $\textbf{x}$. The standardization part of Assumption \ref{assu:vicm_const} can thus be regarded as an identification constraint that removes some of the ambiguity present in the formulation of the model by fixing the locations and scales of the components of $\textbf{z}$. Assumption \ref{assu:vicm_gauss}, on the other hand, is necessitated by the rotational invariance of the multivariate Gaussian distribution. Namely, assume, e.g., that the first two components of $\textbf{z}$ are Gaussian. Then the corresponding subvector is distributionally invariant under rotations and the first two columns of $\boldsymbol{\Omega}$ could be identified only up to a $2 \times 2$ rotation. Thus only a single normally distributed component is allowed. After Assumptions \ref{assu:vicm_const} and \ref{assu:vicm_gauss} we are then left with ambiguity regarding the signs and the order of the independent components which is satisfactory in most applications.


\subsection{Matrix-valued independent component model}
The matrix-valued independent component model is now obtained simply by adding right-hand side mixing to the vector-valued independent component model.

\begin{definition}\label{def:micm}
The matrix-valued independent component model assumes that the observed i.i.d. variables $\textbf{X}_i \in \mathbb{R}^{p \times q}$, $i=1, \ldots ,n$, are realisations of a random matrix $\textbf{X}$ satisfying
\[\textbf{X} = \boldsymbol{\mu} + \boldsymbol{\Omega}_L \textbf{Z} \boldsymbol{\Omega}{}_R^\top, \]
where $\boldsymbol{\mu} \in \mathbb{R}^{p \times q}, \boldsymbol{\Omega}_L \in \mathcal{A}^p$, $\boldsymbol{\Omega}_R \in \mathcal{A}^q$ and the random matrix $\textbf{Z}_i \in \mathbb{R}^{p \times q}$ satisfies Assumptions \ref{assu:micm_const} and \ref{assu:micm_gauss} below.
\end{definition}


\begin{assumptionmat} \label{assu:micm_const}
The components $z_{kl}$ of $\textbf{Z}$ are mutually independent and standardized in the sense that $\E (z_{kl}) = 0$ and $\var (z_{kl}) = 1$.
\end{assumptionmat}

\begin{assumptionmat}\label{assu:micm_gauss}
At most one row of $\textbf{Z}$ consists entirely of Gaussian components and at most one column of $\textbf{Z}$ consists entirely of Gaussian components.
\end{assumptionmat}

The assumptions now guarantee that $\boldsymbol{\Omega}_L$ and $\boldsymbol{\Omega}_R$  are well-defined up to postmultiplication by any matrices $\textbf{PJ}$, $\textbf{P} \in \mathcal{P}^p, \textbf{J} \in \mathcal{J}^p$ or $\textbf{P} \in \mathcal{P}^q, \textbf{J} \in \mathcal{J}^q$, respectively. Thus the first assumption serves again to remove the ambiguity concerning the location of $\textbf{Z}$ and the scales of the columns of $\boldsymbol{\Omega}_L$ and $\boldsymbol{\Omega}_R$, leaving us with the acceptable uncertainty of the signs and order. Again without assumption \ref{assu:micm_gauss}, if, e.g., the first two rows of $\textbf{Z}$ were Gaussian then the first two columns of $\boldsymbol{\Omega}_L$ could be identified only up to a $2\times 2$ rotation. Note that we could still estimate those columns of $\boldsymbol{\Omega}_L$ that correspond to non-Gaussian rows of $\textbf{Z}$ but the successful use of such a method in practice would require some way of estimating or testing for the number of non-Gaussian rows in $\textbf{Z}$. Such a problem is considered for vector-valued ICA in \cite{nordhausen2016asymptotic, nordhausen2017asymptotic} and extending the method to matrix and tensor observations constitutes an interesting future challenge. After the assumptions there is still ambiguity in the proportional sizes of the mixing matrices as the transformations  $\boldsymbol{\Omega}_L\to c \boldsymbol{\Omega}_L$ and $\boldsymbol{\Omega}_R\to c^{-1} \boldsymbol{\Omega}_R$, $c\neq0$, do not change the distribution of $\textbf{X}$. The number of free mixing parameters is therefore $p^2+q^2-1$.


\section{From FOBI to MFOBI}\label{sec:fold}

Taking the same approach as with the independent component models in the previous section, we first review the steps of the classic FOBI procedure, that is, standardization and rotation, for vector-valued data and then suggest a similar procedure for matrix-valued data, called MFOBI, using similar but separate steps from both sides of the matrices.

\subsection{Fourth order blind identification (FOBI)}

Without loss of generality, we assume in the following that the random vector $\textbf{x} \in \mathbb{R}^{p}$ has zero mean, that is, $\boldsymbol{\mu} = \textbf{0}_p $ in the model of Definition \ref{def:vicm}. Note that the following exposition is not the standard way to approach FOBI. However, presenting it this way makes the formulation of MFOBI more intuitive.

We piece together the FOBI-solution by considering the singular value decomposition of the mixing matrix $\boldsymbol{\Omega} = \textbf{U} \textbf{D} \textbf{V}^\tau$, where $\textbf{U}, \textbf{V} \in \mathcal{U}^p$ and $\textbf{D} \in \mathcal{D}^p$ (the diagonal elements of $\textbf{D}$ can be chosen to be positive as the matrix $\boldsymbol{\Omega}$ was assumed to have full rank). The model then has the form
\[\textbf{x} = \textbf{U} \textbf{D} \textbf{V}^\tau \textbf{z}.\]
In this form it is easy to break down the steps in which we gradually  ``lose'' the independence of the components of $\textbf{z}$ and move towards the observed  $\textbf{x}$.
\begin{itemize}
\item[0.] The vector of independent components $\textbf{z}$ has independent components and unit component variances, $\cov ( \textbf{z} ) = \textbf{I}_p$.
\item[1.] The vector of standardized components $\textbf{x}^{st} := \textbf{V}{}^\top \textbf{z}$ has uncorrelated components and unit component variances, $\cov ( \textbf{x}^{st} ) = \textbf{I}_p$.
\item[2.] The vector of uncorrelated components $\textbf{x}^{un} := \textbf{D} \textbf{x}^{st}$ has uncorrelated components, $\cov ( \textbf{x}^{un} ) = \textbf{D}^2$.
\item[3.] The observed vector $\textbf{x} = \textbf{U} \textbf{x}^{un}$ has (generally) correlated components, $\cov ( \textbf{x} ) = \textbf{U} \textbf{D}^2 \textbf{U}^\top$.
\end{itemize}
That is, in Step 1 we lose independence, in the second step the unit variances and finally in the third step the uncorrelatedness. For the solution we then hope to carry out these steps in the reversed order.

\subsubsection{Standardization}

The first step in FOBI consists of standardizing $\textbf{x}$ with an inverse square root of its covariance matrix $\cov (\textbf{x}) =: \textbf{S}$. As
$\textbf{S}= \textbf{U} \textbf{D}^{2} \textbf{U}^\top $ one can choose any matrix  of the form $\textbf{S}^{-1/2}=  \textbf{M}  \textbf{D}^{-1} \textbf{U}^\top$, where $\textbf{M} \in \mathcal{U}^p$. This yields the transformation
\begin{align} \label{theo:vec_stand}
\textbf{x} \mapsto \textbf{S}^{-1/2} \textbf{x} = \textbf{M} \textbf{x}^{st} = \textbf{W} \textbf{z},
\end{align}
where $\textbf{W} := \textbf{M} \textbf{V}^T \in \mathcal{U}^p$. Thus the standardization part moves us directly from $\textbf{x}$ to a standardized random vector and leaves us a rotation away from the independent components.




\subsubsection{Rotation}

To estimate the orthogonal matrix $\textbf{W}^\top$ that rotates the standardized observation in \eqref{theo:vec_stand} to the vector of independent components we use the so-called FOBI-matrix functional, $\textbf{B}(\textbf{x})= \E (\textbf{x} \textbf{x}^\top \textbf{x} \textbf{x}^\top )$. Plugging the standardized vector in, we have
\[\textbf{B}:= \textbf{B}(\textbf{W} \textbf{z})=\textbf{W}  \textbf{B}(\textbf{z})\textbf{W}^\top \]
 where
 \[ \textbf{B}(\textbf{z}) = \E (\textbf{z} \textbf{z}^\top \textbf{z}\textbf{z}^\top )  =   \sum_{k=1}^p (\beta_k + p - 1) \textbf{E}^{kk}  \]
is a diagonal matrix. Therefore,  the orthogonal matrix $\textbf{W}$ can be found from the eigendecomposition of the matrix $\textbf{B}$. However, for the eigenbasis of $\textbf{B}$ to be identifiable, we must make the following assumption that can be seen as a stronger version of Assumption \ref{assu:vicm_gauss}.

\begin{assumptionvec}\label{assu:vicm_asymp}
The kurtosis values $\beta_1, \ldots ,\beta_p$  of the components of $\textbf{z}$ are distinct.
\end{assumptionvec}

The recovering of the independent components by FOBI is then captured by the following formula.
\[\textbf{x} \mapsto \textbf{W}^\top \textbf{S}^{-1/2} \textbf{x}. \]
This process consisting of standardization and rotation will next be translated for matrix-valued observations in an intuitively appealing manner.

\subsection{Matrix fourth order blind identification (MFOBI)}

Without loss of generality, assume that the random matrix $\textbf{X} \in \mathbb{R}^{p \times q}$ has zero mean, that is, $\boldsymbol{\mu} = \textbf{0}_{p \times q} $ in the model of Definition \ref{def:micm}. Resorting again to the singular value decompositions of the full-rank mixing matrices $\boldsymbol{\Omega}_L$ and $\boldsymbol{\Omega}_R$, the model in Definition \ref{def:micm} gets the form
\begin{align*}
\textbf{X} = \boldsymbol{\Omega}_L \textbf{Z} \boldsymbol{\Omega}{}_R^\top = \textbf{U}_{L} \textbf{D}_{L} \textbf{V}_{L}^\top \textbf{Z} \textbf{V}_{R} \textbf{D}_{R} \textbf{U}_{R}^\top.
\end{align*}
Again the diagonal elements of $\textbf{D}_{L}$ and $\textbf{D}_{R}$ can be chosen to be positive.

We then apply a similar analysis for the double mixing process of $\textbf{Z}$ as was done with FOBI previously.
\begin{itemize}
\item[0.] The random matrix  $\textbf{Z}$ has independent components and unit component variances, $\cov \{ \vec (\textbf{Z}) \} = \textbf{I}_{pq}$.
\item[1.] The matrix of standardized components $\textbf{X}^{st} := \textbf{V}_{L}^\top \textbf{Z} \textbf{V}_{R}$ has uncorrelated components and unit component variances, $\cov \{ \vec (\textbf{X}^{st}) \} = \textbf{I}_{pq}$.
\item[2.] The matrix of uncorrelated components $\textbf{X}^{un} := \textbf{D}_{L} \textbf{X}^{st} \textbf{D}_{R}$ has uncorrelated components, $\cov \{ \vec (\textbf{X}^{un}) \} = (\textbf{D}_{R}^2 \otimes \textbf{D}_{L}^2)$.
\item[3.] The observed matrix $\textbf{X} = \textbf{U}_{L} \textbf{X}^{un} \textbf{U}_{R}^\top$ has (generally) correlated components,
$\cov \{ \vec (\textbf{X}) \} = (\textbf{U}_{R} \textbf{D}_{R}^2 \textbf{U}_{R}^\top) \otimes (\textbf{U}_{L}\textbf{D}_{L}^2\textbf{U}_{L}^\top)$.

\end{itemize}

We see that the observed matrix $\textbf{X}$ is built from the matrix of independent components $\textbf{Z}$ in three steps exactly corresponding to the likewise process on random vectors outlined in the section before. Again our objective is to reverse this process.

\subsubsection{Double standardization}

We begin by finding a matrix counterpart for the standardization that provides the first step in FOBI. The presence of a double-sided mixing makes it clear that the standardization has to be performed on $\textbf{X}$ from both left and right. Define the left and right covariance matrices of a zero-mean random matrix $\textbf{X} \in \mathbb{R}^{p \times q}$ as
\[\cov_L ( \textbf{X} ) := \frac{1}{q} \E \left( \textbf{X} \textbf{X}^\top \right) \quad \mbox{and} \quad \cov_R(\textbf{X}) := \frac{1}{p} \E \left( \textbf{X}^\top\textbf{X} \right), \]
The use of $\cov_L$ and $\cov_R$ for matrix observations has been considered already in \cite{srivastava2008models}.
Consider then the left covariance matrix of $\textbf{X}$ in the matrix independent component model of Definition \ref{def:micm},
\begin{align} \label{eq:micm_covxeigen}
\textbf{S}_L := \cov_L ( \textbf{X} ) = \frac{1}{q} \E \left( \textbf{X} \textbf{X}^\top \right) = \frac{1}{q} \textbf{U}_{L} \E \left\{ \textbf{X}^{un} (\textbf{X}^{un})^\top \right\} \textbf{U}_{L}^\top,
\end{align}
where straightforward calculations show that $\E \left\{ \textbf{X}^{un} (\textbf{X}^{un})^\top \right\} = \tra (\textbf{D}_{R}^2) \textbf{D}_{L}^2$
Thus \eqref{eq:micm_covxeigen} provides the eigendecomposition of $\textbf{S}_L$ and all of its inverse square roots are precisely of the form $\sqrt{q} \| \textbf{D}_{R} \|_F^{-1} \textbf{M}_L \textbf{D}_{L}^{-1} \textbf{U}_{L}^\top$, where $\textbf{M}_L \in \mathcal{U}^p$ and $\| \cdot \|_F$ denotes the Frobenius norm. The exact same procedure for the right covariance matrix of $\textbf{X}$ yields
\begin{align*}
\textbf{S}_R := \cov_R ( \textbf{X} ) = \frac{1}{p} \E \left( \textbf{X}^T \textbf{X} \right) = \frac{1}{p} \textbf{U}_{R} \E \left\{ (\textbf{X}^{un})^T \textbf{X}^{un} \right\} \textbf{U}_{R}^\top,
\end{align*}
where $\E \left\{ (\textbf{X}^{un})^T \textbf{X}^{un} \right\} = \tra ( \textbf{D}_{L}^2 ) \textbf{D}_{R}^2$ and the inverse square roots of $\textbf{S}_R$ are precisely of the form $\sqrt{p} \| \textbf{D}_{L} \|_F^{-1} \textbf{M}_R \textbf{D}_{R}^{-1} \textbf{U}_{R}^\top$, where $\textbf{M}_R \in \mathcal{U}^q$.

Using then the inverse square roots of $\textbf{S}_L$ and $\textbf{S}_R$ to doubly standardize the data, we obtain the transformation
\[\textbf{X} \mapsto \textbf{S}_L^{-1/2} \textbf{X} (\textbf{S}_R^{-1/2})^\top = \sqrt{pq} \| \textbf{D}_{L} \|_F^{-1} \| \textbf{D}_{R} \|_F^{-1} \textbf{M}_L \textbf{V}_{L}^\top \textbf{Z} \textbf{V}_{R} \textbf{M}_R^\top. \]
Denoting $\textbf{W}_L := \textbf{M}_L \textbf{V}_{L}^\top \in \mathcal{U}^p$ and $\textbf{W}_R := \textbf{M}_R \textbf{V}_{R}^\top  \in \mathcal{U}^q$ we have the following theorem.

\begin{theorem}\label{theo:micm_stand}
Denote by $\textbf{S}_L^{-1/2}$ and $\textbf{S}_R^{-1/2}$ any inverse square roots of the matrices $\cov_L (\textbf{X})$ and $\cov_R (\textbf{X})$, respectively. Then, under the matrix independent component model of Definition \ref{def:micm},
\[\textbf{S}_L^{-1/2} \textbf{X} (\textbf{S}_R^{-1/2})^\tau \propto \textbf{W}_L \textbf{Z} \textbf{W}_R^\tau, \]
where $\textbf{W}_L \in \mathcal{U}^p$ and $\textbf{W}_R \in \mathcal{U}^q$.
\end{theorem}

Theorem \ref{theo:micm_stand} thus says that the double standardization by $\textbf{S}{}_L^{-1/2}$ and $\textbf{S}{}_R^{-1/2}$ is a natural counterpart of the standardization of a random vector $\textbf{z}$ by $\textbf{S}{}^{-1/2}$, again leaving us only a (double) rotation away from independent components. 

\subsubsection{Double rotation}

We next approach the rotation part with the same mindset. First, notice that we have two logical matrix counterparts for the FOBI functional $\textbf{B}(\textbf{x})$, namely
\begin{align*}
\textbf{B}^0(\textbf{X}) :=  \E \left( \textbf{X}\textbf{X}^\top \textbf{X}\textbf{X}^\top \right) \quad \mbox{and} \quad \textbf{B}^1(\textbf{X}) :=
\E \left( \|\textbf{X}\|_F^2 \textbf{X}\textbf{X}^\top \right),
\end{align*}
both reducing to the ordinary FOBI-matrix functional $\textbf{B}(\textbf{x})$ if $\textbf{X}$ has only one column. For finding the rotations we then use either the pair
\begin{align*} \textbf{B}^0_L := \frac {1}{q}  \textbf{B}^0 \left\{ \textbf{S}_L^{-1/2} \textbf{X} (\textbf{S}_R^{-1/2})^\top \right\} \quad \mbox{and} \quad \textbf{B}^0_R := \frac {1}{p} \textbf{B}^0 \left\{ \textbf{S}_R^{-1/2} \textbf{X}^\top (\textbf{S}_L^{-1/2})^\top \right\} ,
\end{align*}
or the pair
\begin{align*} \textbf{B}^1_L := \frac{1}{q} \textbf{B}^1 \left\{ \textbf{S}_L^{-1/2} \textbf{X} (\textbf{S}_R^{-1/2})^\top \right\} \quad \mbox{and} \quad \textbf{B}^1_R := \frac{1}{p} \textbf{B}^1 \left\{ \textbf{S}_R^{-1/2} \textbf{X}^\top (\textbf{S}_L^{-1/2})^\top \right\}.
\end{align*}
Write next $\tau := \sqrt{pq} \| \textbf{D}_{L} \|_F^{-1} \| \textbf{D}_{R} \|_F^{-1}$, $a_0 :=(p-1)+(q-1) $ and $a_1 := pq - 1$. Plugging in the standardized matrix $\textbf{S}_L^{-1/2} \textbf{X} (\textbf{S}_R^{-1/2})^\top = \tau \textbf{W}_L \textbf{Z} \textbf{W}_R^\top$ we obtain, for $N \in \{0,1\}$,
\begin{align}\label{eq:rotation svd}
\textbf{B}^N_L = \tau^4 \textbf{W}_L \left( a_N \textbf{I}_p + \sum_{k=1}^p \bar{\beta}_{k\bull} \textbf{E}^{kk} \right) \textbf{W}_L^\top \quad \text{and} \quad \textbf{B}^N_R = \tau^4 \textbf{W}_R \left( a_N \textbf{I}_q + \sum_{l=1}^q \bar{\beta}_{\bull l} \textbf{E}^{ll} \right) \textbf{W}_R^\top,
\end{align}
which are precisely the eigendecompositions of the matrices $\textbf{B}_L^N$ and $\textbf{B}_R^N$, giving us a way of finding the missing double rotation by $\textbf{W}_L^\top$ and $\textbf{W}_R^\top$. To identify the needed eigenbases, the matrix counterpart for Assumption \ref{assu:vicm_asymp} is then as follows.

\begin{assumptionmat} \label{assu:micm_asymp}
Both the row averages  $\bar{\beta}_{1\bull}, \ldots ,\bar{\beta}_{p\bull}$ and the column averages
$\bar{\beta}_{\bull 1}, \ldots ,\bar{\beta}_{\bull q}$ of the kurtosis values of $z_{kl}$ are distinct.
\end{assumptionmat}

Interestingly, the number of constraints on the distinctness of the kurtoses of the components does not grow linearly with the number of components but is rather proportional to its square root (assuming the number of rows and the number of columns grow linearly). In a sense MFOBI thus allows for more freedom for the individual marginal distributions. Note that for $q = 1$ Assumption \ref{assu:micm_asymp} reduces to Assumption \ref{assu:vicm_asymp}.

\subsubsection{The method in total}

The similarity between FOBI and MFOBI is now particularly easy to see if we first write the formula for FOBI as
\[\textbf{z} = \textbf{W}^\top \textbf{S}^{-1/2} \textbf{x}, \]
where $\textbf{W}$ has the eigenvectors of $\textbf{B}$ as its columns and the equality sign means equality up to sign-change and permutation. Compare then the above to the same expression for MFOBI:
\[ \textbf{Z} \propto \left( \textbf{W}_L^\top \textbf{S}_L^{-1/2} \right) \textbf{X} \left(\textbf{W}_R^T \textbf{S}_R^{-1/2} \right){}^\top, \]
where $\textbf{W}_L$ and $\textbf{W}_R$ have respectively the eigenvectors of $\textbf{B}_L$ and $\textbf{B}_R$ as their columns and the proportionality is up to permutation and  sign-change from both left and right. Seen this way, MFOBI can simply be regarded as FOBI applied from both sides simultaneously. Recovering the matrix $\textbf{Z}$ only up to proportionality is not a problem as we can always estimate the constant of proportionality using the assumption that $\cov \left\{ \vec (\textbf{Z}) \right\} = \textbf{I}_{pq}$.

\section{Extension to tensor-valued data} \label{section:tensor}

In this section we further extend FOBI to tensor-valued data, producing a method we refer to as TFOBI.
To handle summations over multiple indices we use Einstein's summation convention \citep{mccullagh1987}; that is, whenever an index appears twice, summation over that index is implied.
For example, for a $4$-dimensional tensor $\mbf A = \{ a \lo {ijkl} \}$, the symbol $a \lo {abjk} a \lo {cdjk}$ stands for
\begin{align*}
\tsum \lo j \tsum \lo k  a \lo {abjk} a \lo {cdjk}.
\end{align*}
That is, $\{a \lo {abjk} a \lo {cdjk}\}$ is a 4-dimensional tensor, the $(a,b,c,d)$th entry of which is given above.

\subsection{Tensor independent component model}

Let $\mbf X$ be a random element in $\real \hi {p \lo 1 \times \ldots \times p \lo r}$, that is, a random tensor of order $r$. Following \cite{lathauwer2000}, for a given tensor $\mbf A \in \real \hi {p \lo 1 \times \ldots \times p \lo r}$, we call any $p \lo m$-vector obtained by letting $i \lo m$ vary over $\{1, \ldots, p \lo m \}$ while fixing all the other indices an $m$-mode vector. The term $m$th mode (or ``$m$-mode'') refers to the $m$th direction of a tensor of order $r$, $m=1, \ldots ,r$. In some sense the opposite operation, fixing a value of one of the indices, $i_m = 1, \ldots ,p_m$, while varying the others produces what we call the $m$-mode faces of a tensor. For any given $m=1, \ldots ,r$, a tensor $\mbf A \in \real \hi {p \lo 1 \times \ldots \times p \lo r}$ thus has in total $p_m$ $m$-mode faces of size $p_1 \times \ldots \times p_{m-1} \times p_{m+1} \times \ldots \times p_r$. Notice that the set of $i_m$th elements of all $m$-mode vectors of a tensor $\textbf{A}$ contains the same elements as the $i_m$th $m$-mode face of $\textbf{A}$, $i_m = 1, \ldots ,p_m$.

To work with tensors we next introduce a product operation between a tensor and a matrix that provides a higher order generalization of a linear transformation of a vector by matrix. Following again \cite{lathauwer2000}, for  $\mbf A \in \real \hi { p \lo 1 \times \ldots \times p \lo r}$ and $\mbf B \in \real \hi {q \lo m \times p \lo m}$, let  $\mbf A \odot \lo m \mbf B$  be the
$p \lo 1 \times \ldots \times p \lo {m-1} \times q \lo m \times p \lo {m+1} \times \ldots \times p \lo r$ dimensional tensor whose $(i \lo 1, \ldots, j \lo m, \ldots, i \lo r)$th entry is
\begin{align}\label{eq:tensor matrix}
(\mbf A \odot \lo {m} \mbf B ) \lo {i \lo 1 \ldots j \lo m \ldots i \lo r} =   a \lo {i \lo 1 \ldots i \lo m \ldots i \lo r} b \lo {j \lo m i \lo m}.
\end{align}
Let $\mbf B \lo 1 \in \real \hi {q \lo 1 \times p \lo 1}, \ldots, \mbf B \lo r \in \real \hi {q \lo r \times p \lo r}$. We use the notation  $\mbf A \odot \lo 1 \mbf B \lo 1 \ldots \odot \lo r \mbf B \lo r$ to abbreviate the tensor
\begin{align*}
(\ldots ( \mbf A \odot \lo 1 \mbf B \lo 1) \odot \lo 2 \mbf B \lo 2 \ldots) \odot \lo r \mbf B \lo r) = \{ a \lo {j \lo 1 \ldots j \lo r} b \lo {i \lo 1 j \lo 1} \ldots b \lo {i \lo r j \lo r} \}.
\end{align*}
It is easy to see that for a vector $\textbf{a} \in \mathbb{R}^{p_1}$ we have $(\textbf{a} \odot_1 \textbf{B}_1) = \textbf{B}_1 \textbf{a}$ and for a matrix $\textbf{A} \in \mathbb{R}^{p_1 \times p_2}$ similarly $(\textbf{A} \odot_1 \textbf{B}_1) = \textbf{B}_1 \textbf{A}$ and $(\textbf{A} \odot_2 \textbf{B}_2) = \textbf{A} \textbf{B}_2^\top$, assuming $\textbf{B}_1$ and $\textbf{B}_2$ are of appropriate size. Thus $\odot_m$ can be seen as a linear transformation from the direction of the $m$th mode. Using $m$-mode vectors the multiplication has also a second interpretation; $(\textbf{A} \odot_m \textbf{B}_m)$ applies the linear transformation given by $\textbf{B}_m$ individually to each $m$-mode vector of $\textbf{A}$.

The previous multiplication operation is also commutative in the sense that for distinct values of $m$ the order we apply the multiplications $\odot_m$ has no effect on the outcome. If we want to multiply multiple times from the direction of the same mode commutativity fails and we instead have the following lemma.

\begin{lemma}\label{lemma:tensor matrix product} For any  $\mbf A \in \real \hi { p \lo 1 \times \ldots \times p \lo r}$, $\mbf B \lo 1 \in \real \hi {p \lo 1 \times p \lo 1}, \ldots, \mbf B \lo r \in \real \hi {p \lo r \times p \lo r}$, $\mbf C \lo 1 \in \real \hi {p \lo 1 \times p \lo 1}, \ldots, \mbf C \lo r \in \real \hi {p \lo r \times p \lo r}$, we have
\begin{align}\label{eq:tensor matrix product}
\mbf A \odot \lo 1 (\mbf B \lo 1 \mbf C \lo 1 ) \ldots \odot \lo r (\mbf B \lo r \mbf C \lo r) =
\mbf A \odot \lo 1 \mbf C \lo 1 \ldots \odot \lo r \mbf C \lo r\odot \lo 1 \mbf B \lo 1 \ldots \odot \lo r \mbf B \lo r.
\end{align}
\end{lemma}

\begin{proof}
By definition, the right hand side is the tensor in $\real \hi {p \lo 1 \times \ldots \times p \lo r}$ whose ($i \lo 1 \ldots i \lo r$)th entry is
\begin{align*}
a \lo {k \lo 1 \ldots k \lo r} b \lo {j \lo 1 k \lo 1} \ldots b \lo {j \lo r k \lo r} c \lo {i \lo 1 j \lo 1} \ldots c \lo {i \lo r j \lo r} = a \lo {k \lo 1 \ldots k \lo r} (c \lo {i \lo 1 j \lo 1} b \lo {j \lo 1 k \lo 1})   \ldots (c \lo {i \lo r j \lo r} b \lo {j \lo r k \lo r}).
\end{align*}
The right-hand side is precisely  the ($i \lo 1 \ldots i \lo r$)th entry of the tensor on the left-hand side of \eqref{eq:tensor matrix product}.
\end{proof}

We now have sufficient tools to define the independent component model for tensors.

\begin{definition}\label{def:ticm}
The tensor-valued independent component model assumes that the  observed i.i.d. tensors $\mbf X \lo i \in \mathbb{R}^{p \lo 1 \times \ldots \times p \lo r}$, $i=1, \ldots ,n$, are realisations of a random tensor $\textbf{X}$ satisfying
\begin{align}\label{eq:tic}
 \mbf X = \mbf \mu + \mbf Z \odot \lo 1 \mbf \Omega \lo 1 \ldots \odot \lo r \mbf \Omega \lo r.
\end{align}
where $\boldsymbol{\mu} \in \mathbb{R}^{ p \lo 1 \times \ldots \times p \lo r}$, $\mbf \Omega \lo 1 \in \ca A \hi {p \lo 1}, \ldots, \mbf \Omega \lo r \in \ca A \hi {p \lo r}$,  and $\mbf Z \in \mathbb{R}^{p \lo 1 \times \ldots \times p \lo r}$ satisfies Assumptions \ref{assu:ticm_const} and \ref{assu:ticm_gauss} below.
\end{definition}

\begin{assumptionten} \label{assu:ticm_const}
The components $z_{k \lo 1 \ldots k \lo r}$ of $\textbf{Z}$ are mutually independent and standardized in the sense that $E[z_{k \lo 1 \ldots k \lo r}] = 0$ and $\var [z_{k \lo 1 \ldots k \lo r}] = 1$.
\end{assumptionten}

\begin{assumptionten}\label{assu:ticm_gauss} For each $m = 1, \ldots , r$, at most one $m$-mode face of $\mbf Z$ consists entirely of Gaussian components.
\end{assumptionten}

The above assumptions serve the same purposes as the corresponding assumptions of the vector and matrix independent component models in Section \ref{sec:icm}. The need for Assumption \ref{assu:ticm_gauss} can be seen by considering the product operation $\odot_m \mbf \Omega \lo m$ as a linear transformation of the $m$-mode vectors by $\mbf \Omega \lo m$ and observing that if two or more $m$-mode faces had only Gaussian components the corresponding columns of $\mbf \Omega \lo m$ would be rotationally invariant.

\subsection{The $m$-mode moment matrices of a random tensor}

The matrix unmixing procedure described in Section \ref{sec:fold} involves left and right standardization and then left and right rotation. We need to generalize these to $m$-mode standardization and $m$-mode rotation. 
We first define the $m$-mode product between two tensors: for $\mbf A, \mbf B \in \real \hi {p \lo 1 \times \ldots \times p \lo r}$, the $m$-mode product, $\mbf A \odot \lo {-m} \mbf B$,  is the $p \lo m \times p \lo m$ matrix, the $(s,t)$th entry of which is
\begin{align*}
(\mbf A \odot \lo {-m} \mbf B ) \lo {st} = a \lo {i \lo 1 \ldots i \lo {m-1} \, s \,  i \lo {m+1} \ldots i \lo r} b \lo {i \lo 1 \ldots i \lo {m-1} \, t \, i \lo {m+1} \ldots i \lo r}.
\end{align*}
In some sense the operation $\odot \lo {-m}$ is opposite to the operation $\odot \lo m$ in \eqref{eq:tensor matrix}; whereas $\odot \lo m$ involves the sum over the $m$th index of $\mbf A$, $\odot \lo {-m}$ involves the sum over all indices except the $m$th index of $\mbf A$ and $\mbf B$.

We now extend moment matrices such as
\begin{align*}
\cov \left( \mbf X \right), \quad \cov \left( \mbf X^\top \right), \quad \E \left( \mbf X \mbf X^\top \mbf X \mbf X^\top \right) \quad \mbox{and} \quad \E \left( \mbf X^\top \mbf X \mbf X^\top \mbf X \right)
\end{align*}
to the tensor case.
Again, for convenience and without loss of generality, assume that $E [ \mbf X] = \mbf 0$. As in the matrix case, there are two generalizations of the FOBI functional.

\begin{definition}\label{definition:m-mode moments} The $m$-mode covariance and two types of $m$-mode FOBI functionals of a random tensor $\mbf X \in \real \hi {p \lo 1 \times \ldots \times p \lo r}$ are   the following $p \lo m \times p \lo m$ matrices
\begin{align*}
\cov \loo m ( \mbf X ) =\ali   \left( \tprod \lo {s \ne m} p \lo s \right) \inv \E \left(   \mbf X \odot \lo {-m}   \mbf X  \right),  \\
\mbf B \loo m \hi 0 ( \mbf X) = \ali  \left( \tprod \lo {s \ne m} p \lo s \right) \inv  \E \left\{ ( \mbf X \odot \lo {-m} \mbf X ) \hi 2 \right\} \quad \mbox{and}  \\
\mbf B \loo m \hi 1 ( \mbf X) = \ali  \left( \tprod \lo {s \ne m} p \lo s \right) \inv  \E \left\{ \| \mbf X \|^2 \lo {F} ( \mbf X \odot \lo {-m} \mbf X ) \right\},
\end{align*}
where $\| \cdot \|^2 \lo F$ is the squared Frobenius norm of a tensor (the sum of squared elements).
\end{definition}
Define further
\begin{align}\label{eq:rho m}
\rho \lo m := \tprod \lo {s \ne m} p \lo s.
\end{align}
This proportionality constant reflects the fact that $\odot \lo {-m}$ involves the sum of $\rho \lo m$ terms. 

\subsection{The $m$-mode standardization}

Similar to random matrix unmixing our idea of unmixing a random tensor also consists of two steps: standardization and rotation, except now the two operations have to be performed on each of the $m$ modes of the $r$-tensor.
For each $m = 1, \ldots, r$, let
\begin{align}
\mbf \Omega \lo m = \mbf U \lo m \mbf D \lo m \mbf V \lo m ^\top
\label{eq:svd}
\end{align}
be the singular value decomposition of $\mbf \Omega \lo m$.
The next theorem shows that we can recover $\mbf \Omega \lo m \mbf \Omega \lo m ^\top$ (up to a proportionality constant) from the $m$-mode covariance matrix of $\mbf X$.

\begin{lemma}\label{lemma:m mode tensor covariance} Under the tensor IC model in Definition \ref{def:ticm} we have
\begin{align}
\cov \loo m ( \mbf X ) =  \rho \lo m \inv \left( \textstyle{\prod} \lo {s \ne m} \| \mbf D \lo s \| \hi 2 \lo F \right) \mbf U \lo m \mbf D \lo m \hi 2 \mbf U \lo m ^\top.
\label{eq:above without}
\end{align}
\end{lemma}

\begin{proof}
Without loss of generality, assume $m = 1$. By definition, the $(a,b)$th entry of the matrix
$ \E ( \mbf X \odot \lo {-1} \mbf X ) $ is
\begin{align}\label{eq:ab entry}
\E ( x \lo {a i \lo 2 \ldots i \lo r} x \lo {b  i \lo 2 \ldots i \lo r} ).
\end{align}
By the tensor IC model \eqref{eq:tic} we have
\begin{align*}
x \lo {i \lo 1 \ldots i \lo r} = \omega \lo {i \lo r j \lo r} \hii r \ldots \omega \lo {i \lo 1 j \lo 1} \hii 1 \, z \lo {j \lo 1 \ldots j \lo r}
\end{align*}
where $\omega \lo {ij} \hii m$ are the entries of $\mbf \Omega \lo m$. Hence \eqref{eq:ab entry} can be rewritten as
\begin{align*}
\ali  \E ( \omega \lo {i \lo r j \lo r} \hii r \ldots \omega \lo {a j \lo 1} \hii 1 \, z \lo {j \lo 1 \ldots j \lo r}
 \omega \lo {i \lo r k \lo r} \hii r \ldots \omega \lo {b k \lo 1} \hii 1 \, z \lo {k \lo 1 \ldots k \lo r} ) \\
\ali \hspace{.3in}=  \omega \lo {i \lo r j \lo r} \hii r \ldots \omega \lo {a j \lo 1} \hii 1 \,
\omega \lo {i \lo r k \lo r} \hii r \ldots \omega \lo {b k \lo 1} \hii 1 \, \delta \lo {j \lo 1 k \lo 1} \ldots \delta \lo {j \lo r k \lo r}.
\end{align*}
By the properties of the Kronecker delta we can express the above as
\begin{align*}
\ali   \omega \lo {i \lo r j \lo r} \hii r \ldots \omega \lo {a j \lo 1} \hii 1 \,
\omega \lo {i \lo r j \lo r} \hii r \ldots \omega \lo {b j \lo 1} \hii 1 =  \left( \omega \lo {i \lo 2 j \lo 2} \hii 2 \omega \lo {i \lo 2 j \lo 2} \hii 2 \right)
\ldots \left(  \omega \lo {i \lo r j \lo r} \hii r \omega \lo {i \lo r j \lo r} \hii r \right)
\left( \omega \lo {a j \lo 1} \hii 1 \omega \lo {b j \lo 1} \hii 1 \right)
\end{align*}
which is the $(a,b)$th entry of the matrix $( \textstyle{\prod} \lo {s \ne 1} \| \mbf \Omega \lo s \| \hi 2 \lo F  ) \mbf \Omega \lo 1 \mbf \Omega \lo 1 ^\top$. Now the assertion of the theorem follows from the singular value decomposition \eqref{eq:svd}.
\end{proof}

%
%

Let $\mbf S \lo m := \cov \loo m (\mbf X)$. Relation \eqref{eq:above without} means that
\begin{align*}
\rho \lo m \inv \left( \textstyle{\prod} \lo {s \ne m} \| \mbf D \lo s \| \hi 2 \lo F \right) \mbf U \lo m \mbf D \lo m \hi 2 \mbf U \lo m ^\top
\end{align*}
is in fact the eigendecomposition of $\mbf S \lo m$.
%
%
Thus, all inverse square roots of $\mbf S \lo m$ are of the form
\begin{align*}
\left( \textstyle{\prod} \lo {s \ne m} p \lo s \hi {1/2} \| \mbf D \lo s  \| \lo F \inv \right)  \mbf M \lo m \mbf D \lo m \inv \mbf U \lo m ^\top,  \end{align*}
where $\mbf M \lo m \in \mathcal{U}^{p \lo m} $. We can use these square roots to recover a rotated version of $\mbf Z$, as indicated by the next theorem.


\begin{theorem}\label{theo:ticm_stand}
Let $\mbf S \lo m$ be as defined in the last paragraph.  Then, under the tensor independent component model of Definition \ref{def:ticm},
\begin{align}\label{eq:rotation away}
\mbf X \odot \lo 1 \mbf S \lo 1 \hi {-1/2} \ldots \odot \lo r \mbf S \lo r \hi {-1/2} = \tau
\mbf Z  \odot \lo 1 \mbf W \lo 1 \ldots \odot \lo r \mbf W  \lo r,
\end{align}
where
\begin{align}\label{eq:wm and tau}
\mbf W \lo m := \mbf M \lo m \mbf V \lo m ^\top \in \ca U \hi {p \lo m}, \quad \tau =  \left(\textstyle{\prod} \lo {m=1} \hi r \textstyle{\prod} \lo {s \ne m} p \lo s \hi {1/2} \| \mbf D \lo s  \| \lo F \inv \right),
\end{align}
and $m = 1, \ldots ,r$.
\end{theorem}

\begin{proof}
By Lemma \ref{lemma:tensor matrix product},
\begin{align}
\mbf X \odot \lo 1 \mbf S \lo 1 \hi {-1/2} \ldots \odot \lo r \mbf S   \lo r \hi {-1/2} = \ali
\mbf Z \odot \lo 1 \mbf \Omega \lo 1 \ldots \odot \lo r \mbf \Omega \lo r  \odot \lo 1 \mbf S \lo 1 \hi {-1/2}  \ldots \odot \lo r \mbf S   \lo r \hi {-1/2} \nonumber \\
= \ali
\mbf Z \odot \lo 1 \mbf S \lo 1 \hi {-1/2} \mbf \Omega \lo 1 \ldots \odot \lo r \mbf S \lo r \hi {-1/2}  \mbf \Omega \lo r. \label{eq:above however}
\end{align}
However, we note that
\begin{align*}
\mbf S \lo m \hi {-1/2} \mbf \Omega \lo m = \ali \left( \textstyle{\prod} \lo {s \ne m} p \lo s \hi {1/2} \| \mbf D \lo s  \| \lo F \inv \right)  \mbf M \lo m \mbf D \lo m \inv \mbf U \lo m ^\top \mbf U \lo m \mbf D \lo m \mbf V \lo m ^\top \\
= \ali   \left( \textstyle{\prod} \lo {s \ne m} p \lo s \hi {1/2} \| \mbf D \lo s  \| \lo F \inv \right)  \mbf W \lo m,
\end{align*}
where $\mbf W \lo m := \mbf M \lo m \mbf V \lo m ^\top \in \ca U \hi {p \lo m}$. Substitute the above into \eqref{eq:above however} to prove the desired equality.
\end{proof}

The tensor on the right-hand side of \eqref{eq:rotation away} is only a rotation away from the independent component tensor $\mbf Z$, a step we carry out in the next subsection.

\subsection{The $m$-mode rotation}

Let
\begin{align} \label{eq:bm}
\begin{split}
\textbf{B}_ m \hi 0  &:= \rho \lo m \inv \textbf{B} \loo m ^0(\mbf X \odot \lo 1 \mbf S \lo 1 \hi {-1/2} \ldots \odot \lo r \mbf S \lo r \hi {-1/2} ) \quad \mbox{and} \\
\textbf{B}_ m \hi 1  &:= \rho \lo m \inv \textbf{B} \loo m ^1(\mbf X \odot \lo 1 \mbf S \lo 1 \hi {-1/2} \ldots \odot \lo r \mbf S \lo r \hi {-1/2} )
\end{split}
\end{align}
where $\mbf B \loo m \hi 0$ and $\mbf B \loo m \hi 1$ are the FOBI functionals in Definition \ref{definition:m-mode moments}. In order to manipulate them we need the following lemma.

\begin{lemma}\label{lemma:auuu} Let $\mbf A, \mbf B \in \real \hi {p \lo 1 \times \ldots \times p \lo r}$, $\mbf U \lo 1 \in \ca U \hi {p \lo 1},  \ldots , \mbf U \lo r \in \ca U \hi {p \lo r}$. Then
\begin{align}\label{eq:above without 2}
(\mbf A \odot \lo 1 \mbf U \lo 1 \ldots \odot \lo r \mbf U \lo r) \odot \lo {-m} (\mbf B \odot \lo 1 \mbf U \lo 1 \ldots \odot \lo r \mbf U \lo r) = \mbf U \lo m ( \mbf A \odot \lo {-m} \mbf B) \mbf U \lo m ^\top.
\end{align}
\end{lemma}

\begin{proof}
Without loss of generality assume that $m=1$. The $(a,b)$th entry of the matrix on the left-hand side of \eqref{eq:above without 2} is
\begin{align*}
\ali (\mbf A \odot \lo 1 \mbf U \lo 1 \ldots \odot \lo r \mbf U \lo r) \lo {a \, i \lo 2 \ldots i \lo r} (\mbf B \odot \lo 1 \mbf U \lo 1 \ldots \odot \lo r \mbf U \lo r) \lo {b \, i \lo 2 \ldots i \lo r} \\
\ali \hspace{.5in} = a \lo {j \lo 1 \ldots j \lo r} u \lo {a \, j \lo 1} \hii 1 u \lo {i \lo 2 \, j \lo 2} \hii 2 \ldots u \lo {i \lo r j \lo r} \hii r
b \lo {k \lo 1 \ldots k \lo r} u \lo {b \, k \lo 1} \hii 1 u \lo {i \lo 2 \, k \lo 2} \hii 2 \ldots u \lo {i \lo r k \lo r} \hii r \\
\ali \hspace{.5in} = a \lo {j \lo 1 \ldots j \lo r}b \lo {k \lo 1 \ldots k \lo r} (u \lo {i \lo 2 \, j \lo 2} \hii 2 u \lo {i \lo 2 \, k \lo 2} \hii 2) \ldots ( u \lo {i \lo r j \lo r} \hii r u \lo {i \lo r k \lo r} \hii r )
 u \lo {a \, j \lo 1} \hii 1  u \lo {b \, k \lo 1} \hii 1 \\
 \ali \hspace{.5in} = a \lo {j \lo 1 \ldots j \lo r}b \lo {k \lo 1 \ldots k \lo r}  \delta \lo {j \lo 2 k \lo 2}   \ldots \delta \lo {j \lo r k \lo r}
 u \lo {a \, j \lo 1} \hii 1  u \lo {b \, k \lo 1} \hii 1 .
\end{align*}
The above reduces to
\begin{align*}
 a \lo {j \lo 1 j \lo 2 \ldots j \lo r}b \lo {k \lo 1 j \lo 2 \ldots j \lo r}
 u \lo {a \, j \lo 1} \hii 1  u \lo {b \, k \lo 1} \hii 1,
\end{align*}
which is the $(a,b)$th entry of the matrix on the right-hand side of \eqref{eq:above without 2}.
\end{proof}


Define the $m$-flattening, or $m$-unfolding, of a tensor $\textbf{A} \in \mathbb{R}^{p_1 \times \ldots \times p_r}$ to be the matrix $\textbf{A}_{(m)} \in \mathbb{R}^{p_m \times \rho_m}$ obtained by taking all the $m$-mode vectors of $\textbf{A}$ and stacking them horizontally into a matrix. As for the order of stacking we choose to use the cyclical unfolding described in \cite{lathauwer2000}. Then, for $\textbf{A}^* := \textbf{A} \odot_1 \textbf{B}_1 \ldots \odot_r \textbf{B}_r$, we have
\begin{align}\label{eq:tensor_flatten_kron}
\textbf{A}^*_{(m)} = \textbf{B}_m \textbf{A}_{(m)} \left( \textbf{B}_{m+1} \otimes \ldots \otimes \textbf{B}_r \otimes \textbf{B}_1 \otimes \ldots \otimes \textbf{B}_{m-1} \right).
\end{align}

Flattening can also be used to express the $m$-mode product of a tensor with itself with means of ordinary matrix multiplication. Namely,
\begin{align}\label{eq:tensor_flatten_matrix}
\mbf A \odot \lo {-m} \mbf A = \mbf A \loo m \mbf A {} \loo m ^\top.
\end{align}

For a tensor $\mbf A \in \real \hi {p \lo 1 \times \ldots \times p \lo r}$ let $\bar {\mbf A} \lo {-m}$ be the $p_m$-vector whose $i_m$th element is the mean of the $\rho_m$ elements of the $i_m$th $m$-mode face of $\textbf{A}$, $i_m = 1, \ldots ,p_m$. Expressed via the previously defined $m$-flattening $\bar {\mbf A} \lo {-m}$ thus contains the row means of $\textbf{A}_{(m)}$.

The next theorem shows that the rotations $\mbf W \lo m$ can be recovered from the eigendecompositions of $\mbf B \lo m \hi 0$ and $\mbf B \lo m \hi 1$.

\begin{theorem}\label{theorem:tensor rotation svd}
Let $\rho \lo m$,  $\mbf W \lo m$, $\tau$,   $\mbf B \lo m \hi 0$ and $\mbf B \lo m \hi 1$  be as defined in \eqref{eq:rho m}, \eqref{eq:wm and tau}, and \eqref{eq:bm}. Let $\mbf \beta \in \real \hi {p \lo 1 \times \ldots \times p \lo r}$ be the tensor with the entries $\E ( z \lo {i \lo 1 \ldots i \lo r} \hi 4 )$. Then
\begin{align*}
\mbf B \lo m \hi 0 = \ali \tau^4 \textbf{W}_m  \left\{ (p \lo m - 1 + \rho \lo m - 1) \mbf I \lo {p \lo m}  + \mathrm{diag} ( \bar { \mbf \beta } \lo {-m} )  \right\} \textbf{W}_m ^\top,  \\
\textbf{B} \lo m \hi 1 =\ali \tau^4 \textbf{W}_m  \left\{ (p \lo m \rho \lo m - 1)  \textbf{I}_{p \lo m}   + \mathrm{diag} ( \bar { \mbf \beta } \lo {-m} )  \right\} \textbf{W}_m ^\top,
\end{align*}
where $\mathrm{diag} ( \bar { \mbf \beta } \lo {-m} )$ is the diagonal matrix having the elements of $\bar { \mbf \beta } \lo {-m} \in \mathbb{R}^{p_m}$ on its diagonal.
\end{theorem}

\begin{proof}
Again, without loss of generality, assume $m = 1$. By definition,
\begin{align*}
\mbf B \lo 1 \hi 0 = \rho \lo 1 \inv  \E \left[ \left\{ ( \mbf X \odot \lo 1 \mbf S \lo 1 \hi {-1/2} \ldots \odot \lo r \mbf S \lo r \hi {-1/2}) \odot \lo {-1}
( \mbf X \odot \lo 1 \mbf S \lo 1 \hi {-1/2} \ldots \odot \lo r \mbf S \lo r \hi {-1/2}) \right\} \hi 2 \right].
\end{align*}
By Theorem \ref{theo:ticm_stand}, the right-hand side is
\begin{align*}
\rho \lo 1 \inv \tau \hi 4 \E \left[ \left\{ ( \mbf Z \odot \lo 1 \mbf W \lo 1 \ldots \odot \lo r \mbf W \lo r )\odot \lo {-1}
( \mbf Z \odot \lo 1 \mbf W \lo 1 \ldots \odot \lo r \mbf W \lo r ) \right\} \hi 2 \right].
\end{align*}
By Lemma \ref{lemma:auuu}, this is
$
\rho \lo 1 \inv \tau \hi 4 \mbf W \lo 1  \E \left\{ ( \mbf Z  \odot \lo {-1}
  \mbf Z  )\hi 2 \right\} \mbf W \lo 1 ^\top.
$
and by \eqref{eq:tensor_flatten_matrix} the expectation can be expressed as
\begin{align*}
\E \left\{ ( \mbf Z  \odot \lo {-1}
  \mbf Z  )\hi 2 \right\} = \E \left( \mbf Z \loo 1 \mbf Z \loo 1 ^\top \mbf Z \loo 1 \mbf Z \loo 1 ^\top \right).
\end{align*}
Now applying the matrix identities in \eqref{eq:rotation svd} completes the proof for $\mbf B \lo m \hi 0$. The proof for $\mbf B \lo m \hi 1$ is carried out similarly by reducing the matter into the matrix case.
\end{proof}

Theorem \ref{theorem:tensor rotation svd} says that $\mbf W \lo m$ has the eigenvectors of $\mbf B \lo m \hi 0$ and $\mbf B \lo m \hi 1$ as its columns.
In other words, we can recover the orthogonal matrices $\mbf W \lo m$ from the eigendecompositions of $\mbf B \lo m \hi 0$ or $\mbf B \lo m \hi 1$, $m=1, \ldots ,r$. Again, to identify the  eigenbases we need the following assumption.

\begin{assumptionten} \label{assu:ticm_asymp}
For each $m = 1, \ldots, r$, the components of $\bar {\mbf \beta} \lo {-m}$ are distinct.
\end{assumptionten}

The next corollary puts the $m$-mode standardizations and rotations together to recover the independent component from a random tensor $\mbf X$.

\begin{corollary} Let $\mbf S \lo m \hi {-1/2}$ be any square root of $\mbf S \lo m$ and let $\mbf W \lo m$ have the eigenvectors of either $\mbf B \lo m \hi 0$ or $\mbf B \lo m \hi 1$ as its columns, $m = 1, \ldots ,r$. Then, under Assumptions \ref{assu:ticm_const} and \ref{assu:ticm_asymp}, we have
\begin{align*}
\mbf X \odot \lo 1 (\mbf W \lo 1 ^\top  \mbf S \lo 1 \hi {-1/2} ) \ldots \odot \lo r ( \mbf W \lo r  ^\top \mbf S \lo r \hi {-1/2} ) = \tau \mbf Z.
\end{align*}
\end{corollary}

\begin{proof}
Multiply both sides of the equation \eqref{eq:rotation away} from the right by
\begin{align*}
\odot \lo 1 \mbf W \lo 1 ^\top \ldots \odot \lo r \mbf W \lo r ^\top
\end{align*}
and evoke tensor-matrix product rule in Lemma \ref{lemma:tensor matrix product} to prove the result.
\end{proof}

\section{Limiting distributions}\label{sec:asymp}

In this section we pursue the asymptotic distributions of the unmixing estimates given by the extended ICA procedures in the previous sections. We will focus primarily on MFOBI because the corresponding results for TFOBI follow directly from the results for MFOBI, as detailed in Remark \ref{rem:tfobi}. However, we first discuss the important concept of equivariance.

\subsection{Equivariance and independent component functionals}

In the vector-valued case for example \cite{MiettinenNordhausenOjaTaskinen:2014} state  that an unmixing functional $\boldsymbol{\Gamma}$ must satisfy the following two  conditions.
(i) For a distribution of $\bo z \in \mathbb{R}^p$ with standardized and mutually independent components, $\boldsymbol{\Gamma}(\bo z)=\bo I_p$  and, (ii) for the distribution of any $\bo x \in \mathbb{R}^p$, it holds that $\boldsymbol{\Gamma}(\bo A \bo x) = \boldsymbol{\Gamma}(\bo x)\bo A^{-1}$, for all $\bo A \in \mathcal{A}^p$ (in both conditions the equalities are understood up to permutation and sign changes of the rows). The second condition means that the functional is equivariant under affine transformations and $\boldsymbol{\Gamma}( \bo x)\bo x$ is thus independent of the used coordinate system. Theoretical derivations can then be limited to the case $\bs \Omega = \bo I_p$.

Consider next the unmixing matrix functionals in the tensor case and write $\boldsymbol{\Gamma}_{(m)}(\bo X) := \mbf W \lo m ^\top  \mbf S \lo m \hi {-1/2}$ for the $m$-mode unmixing matrix functional, $m=1, \ldots ,r$.
The functional $\boldsymbol{\Gamma}_{(m)}$ is said to be \textit{(fully)  affine equivariant} if, for all $\bo X\in \mathbb{R}^{p_1\times \ldots \times p_r} $ and  all $\textbf{A}_1 \in \mathcal{A}^{p_1}, \ldots , \textbf{A}_r \in \mathcal{A}^{p_r}$,
\[
\boldsymbol{\Gamma}_{(m)}(\bo X \odot_1 \textbf{A}_1 \ldots \odot_r \textbf{A}_r)=\boldsymbol{\Gamma}_{(m)}(\bo X) \textbf{A}_m^{-1}.
\]
This is however true for our unmixing matrix functionals only if $\textbf{A}_1, \ldots ,\textbf{A}_r$ are all orthogonal. The TFOBI unmixing matrix functionals $\boldsymbol{\Gamma}_{(m)}$ are thus \textit{orthogonally equivariant}. Also the weaker \textit{marginal affine equivariance} 
\[
\boldsymbol{\Gamma}_{(m)}(\bo X \odot_m \textbf{A}_m)=\boldsymbol{\Gamma}_{(m)}(\bo X)\textbf{A}_m^{-1},
\]
for some fixed $m = 1, \ldots ,r$, holds only if all $\textbf{A}_s$, $s \neq m$ are orthogonal. The reason why both of these conditions fail in the general case is that the $m$-mode covariance functionals are not fully affine equivariant in the sense that, for all $\bo X\in \mathbb{R}^{p_1\times \ldots \times p_r} $ and  all $\textbf{A}_1 \in \mathcal{A}^{p_1}, \ldots , \textbf{A}_r \in \mathcal{A}^{p_r}$,
\begin{align}\label{eq:cov_ae}
\cov_{(m)}(\bo X \odot_1 \textbf{A}_1 \ldots \odot_r \textbf{A}_r)=\textbf{A}_m  \cov_{(m)}(\bo X) \textbf{A}_m^\top, \quad \forall m = 1, \ldots ,r.
\end{align}
The condition \eqref{eq:cov_ae} also holds only if $\textbf{A}_1, \ldots ,\textbf{A}_r$ are all orthogonal, leading then into the orthogonal equivariance and marginal orthogonal equivariance of $\boldsymbol{\Gamma}_{(m)}$. In fact, \eqref{eq:cov_ae} in general seems such strict a requirement that we conjecture that no functional satisfying it exists. This would then imply also that no fully affine equivariant tensor unmixing matrix functionals based on separate standardization and rotation steps exist. Note however, that marginally affine equivariant $\boldsymbol{\Gamma}_{(m)}$ for a single direction can be obtained if $\cov_{(m)}$ and then $\bo B_{(m)}^0$ or $\bo B_{(m)}^1$ are applied separately for each direction.

The lack of full affine equivariance  means that the asymptotic results for the unmixing matrix estimates for general $\boldsymbol{\Omega}_L$ and $\boldsymbol{\Omega}_R$ no longer follow from the results in the simple case, $\boldsymbol{\Omega}_L = \textbf{I}_p$, $\boldsymbol{\Omega}_R = \textbf{I}_q$, and thus the comparison of different estimates becomes difficult. 
In the following we find the limiting distributions of the FOBI estimate $\hat{\boldsymbol{\Gamma}}$ and the MFOBI estimates $\hat{\boldsymbol{\Gamma}}_L$ and $\hat{\boldsymbol{\Gamma}}_R$ under the assumptions that $\boldsymbol{\Omega} = \textbf{I}_{p}$ (FOBI) and that $\boldsymbol{\Omega}_L = \textbf{I}_p$ and $\boldsymbol{\Omega}_R = \textbf{I}_q$ (MFOBI). The estimates are obtained by applying the functionals to empirical distributions of sample size $n$.



\subsection{Limiting distribution of the FOBI estimate}

The asymptotic behavior of the classic FOBI was first derived in \cite{IlmonenNevalainenOja2010} and requires Assumption \ref{assu:vicm_asymp} on the distinct kurtosis values of the components. The following results are however in the form of \cite{miettinen2014fourth}, see their Theorem 8 and Corollary 3.

\begin{theorem}\label{theo:vicm_asymp}
Let $\textbf{z}_1, \ldots ,\textbf{z}_n$ be a random sample from a $p$-variate distribution having finite eighth moments and satisfying assumptions \ref{assu:vicm_const} and \ref{assu:vicm_asymp}. Assume further that $\boldsymbol{\Omega} = \textbf{I}_p$ and that the standardization functional $\hat{\mbf S}{} \hi {-1/2}$ is chosen to be symmetric. Then there exists a sequence of FOBI estimates such that $\hat{\boldsymbol{\Gamma}}\rightarrow_P \textbf{I}_p$ and
\begin{align*}
\sqrt{n} (\hat{\gamma}_{kk} - 1) &= -\dfrac{1}{2} \sqrt{n} (\hat{s}_{kk} - 1) + o_P(1), \\
\sqrt{n} \hat{\gamma}_{kk'} &= \dfrac{\sqrt{n} \hat{Q} - (\beta_k + p + 1) \sqrt{n} \hat{s}_{kk'}}{\beta_k - \beta_{k'}}
 + o_P(1),
\end{align*}
where $\hat{Q} = \hat{q}_{kk'} + \hat{q}_{k'k} + \sum_{m \neq k,k'} \hat{q}_{mkk'}$ and $k \neq k'$.
\end{theorem}

Based on Theorem \ref{theo:vicm_asymp} we can then compute the asymptotic variances of the elements of the estimated unmixing matrix $\hat{\boldsymbol{\Gamma}}$.

\begin{corollary}\label{cor:vicm_asymp}
Under the assumptions of Theorem \ref{theo:vicm_asymp} the limiting distribution of $\sqrt{n} \, \vec (\hat{\boldsymbol{\Gamma}} - \textbf{I}_p)$ is multivariate normal with mean vector $\textbf{0}_{p^2}$ and the following asymptotic variances.
\begin{alignat*}{3}
&ASV&(\hat{\gamma}_{kk}) &= \frac{\beta_k - 1}{4}, &\\
&ASV&(\hat{\gamma}_{kk'}) &= \frac{\omega_k + \omega_{k'} - \beta_k^2 - 6 \beta_{k'} + 9 + \sum_{m \neq k,k'} (\beta_m - 1)}{(\beta_k - \beta_{k'})^2}, &\quad k \neq k'.
\end{alignat*}
\end{corollary}

As Corollary \ref{cor:vicm_asymp} shows, the asymptotic variance of any off-diagonal element $\gamma_{kk'}$ of the unmixing matrix depends also on components other than $z_k$ and $z_{k'}$ (via their kurtoses). Of the commonly used independent component analysis methods, FastICA, FOBI and JADE, FOBI is unique in this sense, partly explaining its inferiority to the other methods.

\subsection{Limiting distribution of the MFOBI estimate}

We provide the asymptotic properties of only the left-hand side unmixing matrix estimate $\hat{\boldsymbol{\Gamma}} := \hat{\boldsymbol{\Gamma}}_L$, the right-hand side version being again easily obtained by reversing the roles of rows and columns. Here $N=0$ or $N=1$ depending on the choice of the FOBI functional and the sample left and right covariance matrices are denoted by $\bar{\mbf S}{}_L := (\bar{s}{}^L_{kk'})$ and $\bar{\mbf S}{}_R := (\bar{s}{}^R_{kk'})$

\begin{theorem}\label{theo:micm_asymp}
Let $\textbf{Z}_1, \ldots ,\textbf{Z}_n$ be a random sample from a distribution of a matrix-valued $\textbf{Z} \in \mathbb{R}^{p \times q}$ having finite eighth moments and satisfying assumptions \ref{assu:micm_const} and \ref{assu:micm_asymp}. Assume further that $\boldsymbol{\Omega}_L = \textbf{I}_p$ and $\boldsymbol{\Omega}_R = \textbf{I}_q$, and that the left and right standardization functionals, $\bar{\mbf S}{} \lo L \hi {-1/2}$ and $\bar{\mbf S}{} \lo R \hi {-1/2}$, are chosen to be symmetric. Then there exists a sequence of left MFOBI estimates such that  $\hat{\boldsymbol{\Gamma}}\rightarrow_P \textbf{I}_p$ and
\begin{align*}
\sqrt{n} (\hat{\gamma}_{kk} - 1) &= -\dfrac{1}{2} \sqrt{n} (\bar{s}_{kk} - 1)   + o_P(1), \\
\sqrt{n} \hat{\gamma}_{kk'} &= \frac{\sqrt{n} \bar{Q} + \sqrt{n} \bar{R}^N - (\bar{\beta}_{k\bull} + b_N) \sqrt{n} \bar{s}_{kk'} }{(\bar{\beta}_{k\bull} - \bar{\beta}_{k'\bull})} + o_P(1), &\quad k \neq k',
\end{align*}
where $\bar{Q} = \bar{q}_{kk'} + \bar{q}_{k'k} + \sum_{m \neq k,k'} \bar{q}_{mkk'}$, $\bar{R}^N = \bar{r}_{kk'} + \bar{r}_{k'k} + \sum_{m \neq k,k'} \bar{r}^N_{mkk'}$, $b_0 = 2q + p - 1$ and $b_1 = qp + 1$.
\end{theorem}

\begin{corollary}\label{cor:micm_asymp}
i) Under the assumptions of Theorem \ref{theo:micm_asymp} the limiting distribution of $\sqrt{n} \, \vec (\hat{\boldsymbol{\Gamma}} - \textbf{I}_p)$ is multivariate normal with mean vector $\textbf{0}_{p^2}$ and the following asymptotic variances.
\begin{alignat*}{3}
&ASV&(\hat{\gamma}_{kk}) &= \frac{\bar{\beta}_{k\bull} - 1}{4q}, &\\
&ASV&(\hat{\gamma}_{kk'}) &= \frac{\bar{\omega}_{k\bull} + \bar{\omega}_{k'\bull} -\bar{\beta}_{k\bull}^2 + 2 \delta_{kl} + (q - 1) \bar{\beta}_{k\bull} + (q - 7) \bar{\beta}_{k'\bull} + c_N}{q(\bar{\beta}_{k\bull} - \bar{\beta}_{k'\bull})^2}, &\quad k \neq k',
\end{alignat*}
where $c_0 = \sum_{m \neq k k'} \bar{\beta}_{m\bull} + pq - 2p - 4q + 15$ and $c_1 = q \sum_{m \neq k k'} \bar{\beta}_{m\bull} - pq + 11$.
\end{corollary}

\begin{proof}

The proof for the consistency of the estimator is obtained similarly as in the proof of Theorem 5.1.1 in \cite{virta2015joint}.

Write then
\begin{alignat*}{2}
\bar{\textbf{L}} = (\bar{l}_{kk'}) := \bar{\textbf{S}}{}_L^{-1/2} \rightarrow_P \textbf{I}_p, \quad & \bar{\textbf{L}}{}^* := \bar{\textbf{L}}{}^\top \bar{\textbf{L}} \rightarrow_P \textbf{I}_p, \\
\bar{\textbf{R}} = (\bar{r}_{ll'}) := \bar{\textbf{S}}{}_R^{-1/2} \rightarrow_P \textbf{I}_q, \quad & \bar{\textbf{R}}{}^* := \bar{\textbf{R}}{}^\top \bar{\textbf{R}} \rightarrow_P \textbf{I}_q.
\end{alignat*}
Limiting normal distributions of the components of the sample covariance functionals imply that $\sqrt{n}(\bar{\textbf{S}}_L - \textbf{I}_p) = O_P(1)$ and $\sqrt{n}(\bar{\textbf{S}}_R - \textbf{I}_q) = O_P(1)$ and the following two asymptotic expansions are then easy to prove using Slutsky's theorem, see, e.g., the supplementary material to \cite{virta2015joint}.
\begin{align*}
\sqrt{n}(\bar{\textbf{L}} - \textbf{I}_p) = & -\frac{1}{2} \sqrt{n} (\bar{\textbf{S}}{}_L - \textbf{I}_p) + o_P(1), \\
\sqrt{n} (\bar{\textbf{L}}{}^\top \bar{\textbf{L}} - \textbf{I}_p ) = &  \sqrt{n} (\bar{\textbf{L}} - \textbf{I}_p) + \sqrt{n} (\bar{\textbf{L}}{}^\top - \textbf{I}_p) + o_P(1).
\end{align*}

The estimated left unmixing functional is then $\hat{\boldsymbol{\Gamma}} := \hat{\textbf{W}}{}^\top_L \bar{\textbf{L}}$, where $\hat{\textbf{W}}{}^\top_L$ is obtained from the eigendecomposition of the sample left FOBI functional $\bar{\textbf{B}}{}_L^N = (\bar{b}^N_{kk'}) = \hat{\textbf{W}}_L \hat{\boldsymbol{\Lambda}}{}_L^N \hat{\textbf{W}}{}^\top_L$, where $\hat{\boldsymbol{\Lambda}}{}_L^N \rightarrow_P \boldsymbol{\Lambda}{}_L^N$. The asymptotic behavior of the diagonal elements $\sqrt{n} \hat{\gamma}_{kk}$ of the estimated left unmixing functional can be derived similarly as in the proof of Theorem 4.1.2 of \cite{virta2015joint}. For the off-diagonal elements, using Slutsky's theorem and the fact that $\hat{\boldsymbol{\Lambda}}{}_L^N$ is diagonal, it is straightforward to show that we have for an arbitrary $(k, k')$-element of the estimated left unmixing functional
\begin{align}\label{eq:proofs_asymp}
\sqrt{n} \hat{\gamma}_{kk'} = \frac{\sqrt{n} \bar{b}{}^N_{kk'} + (\bar{\beta}_{k\bull} - \bar{\beta}_{k'\bull}) \sqrt{n} \bar{l}_{kk'}}{\bar{\beta}_{k\bull} - \bar{\beta}_{k'\bull}} + o_P(1), \quad k \neq k'.
\end{align}

The problem lies then in finding the asymptotic behavior of an arbitrary off-diagonal element $\sqrt{n} \hat{b}{}^N_{kk'}$. Consider first the case $N=0$ and write $\bar{\textbf{B}}{}^0_L$ open according to its definition:
\[\sqrt{n} \left( \bar{\textbf{B}}{}_L^0 - \boldsymbol{\Lambda}{}^0_L \right) = \bar{\textbf{L}} \left( \sqrt{n} \frac{1}{n} \sum_{i=1}^n \tilde{\textbf{Z}}_i \bar{\textbf{R}}{}^* \tilde{\textbf{Z}}{}_i^\top  \bar{\textbf{L}}{}^* \tilde{\textbf{Z}}_i \bar{\textbf{R}}{}^* \tilde{\textbf{Z}}{}_i^\top  \right) \bar{\textbf{L}}{}^\top - \sqrt{n} \boldsymbol{\Lambda}{}^0_L, \]
where $\tilde{\textbf{Z}}_i := \textbf{Z}_i - \bar{\textbf{Z}}$. Inspecting a single off-diagonal element yields
\[\sqrt{n} \bar{b}{}^0_{kk'} = \frac{1}{q} \sum_{defgstuv} \left( \sqrt{n} \, \bar{l}_{kd} \bar{r}{}^*_{ef} \bar{l}{}^*_{gs} \bar{r}{}^*_{tu} \bar{l}_{k'v} \frac{1}{n} \sum_{i=1}^n \tilde{z}_{i, de} \tilde{z}_{i, gf} \tilde{z}_{i, st} \tilde{z}_{i, vu}\right). \]
Next, expand each of the covariance terms one-by-one starting with $\sqrt{n} \bar{l}_{kd} = \sqrt{n} (\bar{l}_{kd} - \delta_{kd}) + \sqrt{n} \delta_{kd}$. After each expansion the first term has a multiplicand that is $O_P(1)$ and Slutsky's theorem guarantees the convergence of the corresponding product. Note also that
\[\frac{1}{n} \sum_{i=1}^n \tilde{z}_{i, de} \tilde{z}_{i, gf} \tilde{z}_{i, st} \tilde{z}_{i, vu} \rightarrow_P \E \left( z_{de} z_{gf} z_{st} z_{vu} \right). \]
The number of sums decreases at each step finally resulting into
\begin{align*}
\sqrt{n} \bar{b}{}^0_{kk'} &= (2q + p - 1 + \bar{\beta}_{k'\bull}) \sqrt{n} \hat{l}_{kk'} + (2q + p - 1 + \bar{\beta}_{k\bull}) \sqrt{n} \hat{l}_{k'k} \\
&+ \sum_{egt} \sqrt{n} \frac{1}{n} \sum_{i=1}^n \tilde{z}_{i, ke} \tilde{z}_{i, ge} \tilde{z}_{i, gt} \tilde{z}_{i, k't} + o_P(1),
\end{align*}
the last proper term of which partitions into the quantities defined in Section \ref{sec:nota} as
\[\sqrt{n}\bar{q}_{kk'} + \sqrt{n}\bar{q}_{k'k} + \sum_{m \neq k,k'} \sqrt{n}\bar{q}_{mkk'} + \sqrt{n} \bar{r}_{kk'} + \sqrt{n} \bar{r}_{k'k} + \sum_{m \neq k,k'} \sqrt{n} \bar{r}^0_{mkk'},\]
after which plugging everything into expression \eqref{eq:proofs_asymp} gives the desired result.

The proof for the case $N=1$ is almost similar, only the starting expression is somewhat different:
\[ \sqrt{n} \bar{b}{}^1_{kk'} = \frac{1}{q} \sum_{defghstuv}  \sqrt{n} \bar{l}_{kd} \bar{r}{}^*_{ef} \bar{l}_{k'g} \bar{l}_{hs} \bar{r}{}^*_{tu} \bar{l}_{hv} \frac{1}{n} \sum_{i=1}^n \tilde{z}_{i, de} \tilde{z}_{i, gf} \tilde{z}_{i, st} \tilde{z}_{i, vu}. \]
For both choices of $N$ the asymptotic variances of Corollary \ref{cor:micm_asymp} are then straightforward, albeit a bit tedious, to compute using both Tables \ref{tab:nota_varcov1} and \ref{tab:nota_varcov2} containing covariances between the different terms in addition to the following covariances not fitting into the tables: $nq \cdot \cov [\bar{q}_{mkk'}, \bar{q}_{m'kk'}] = 1$, $nq \cdot \cov [\bar{q}_{mkk'}, \bar{r}^0_{m'kk'}] = 0$, $nq \cdot \cov [\bar{q}_{mkk'}, \bar{r}^1_{m'kk'}] = q^*$, $nq \cdot \cov [\bar{r}^0_{mkk'}, \bar{r}^0_{m'kk'}] = 0$ and $nq \cdot \cov [\bar{r}^1_{mkk'}, \bar{r}^1_{m'kk'}] = q^{*2}$, where $m \neq m'$ and $q^* := q - 1$.

\begin{table}[h]
\centering
\caption{Covariances of $\sqrt{nq}$ times the row and column quantities, $k \neq k' \neq m$.}
\begin{tabular}{l||c|c|c}
 & $\bar{q}_{kk'}$ & $\bar{q}_{k'k}$ & $\bar{q}_{mkk'}$  \\ \hline \hline
 & & & \\[-2ex]
$\bar{q}_{kk'}$ & $\bar{\omega}_{k\bull}$ & $\delta_{kk'} + \bar{\beta}_{k\bull} \bar{\beta}_{k'\bull}$ & $\bar{\beta}_{k\bull}$ \\
$\bar{q}_{k'k}$ & $-$ & $\bar{\omega}_{k'\bull}$ & $\bar{\beta}_{k'\bull}$ \\
$\bar{q}_{mkk'}$ & $-$ & $-$ & $\bar{\beta}_{m\bull}$
\end{tabular}
\label{tab:nota_varcov1}
\end{table}

\begin{table}[h]
\centering\caption{Covariances of $\sqrt{nq}$ times the row and column quantities, $k \neq k' \neq m$ and $q^* := q - 1$.}
\begin{tabular}{l||c|c|c|c|c}
 & $\bar{r}_{kk'}$ & $\bar{r}_{k'k}$ & $\bar{r}^0_{mkk'}$ & $\bar{r}^1_{mkk'}$ & $\bar{s}_{kk'}$ \\ \hline \hline
 & & & & & \\[-2ex]
$\bar{q}_{kk'}$ & $q^*\bar{\beta}_{k\bull}$ & $q^*\bar{\beta}_{k\bull}$ & $0$ & $q^*\bar{\beta}_{k\bull}$ & $\bar{\beta}_{k\bull}$ \\
$\bar{q}_{k'k}$ & $q^*\bar{\beta}_{k'\bull}$ & $q^*\bar{\beta}_{k'\bull}$ & $0$ & $q^*\bar{\beta}_{k'\bull}$ & $\bar{\beta}_{k'\bull}$ \\
$\bar{q}_{mkk'}$ & $q^*$ & $q^*$ & $0$ & $q^*$ & $1$ \\
$\bar{r}_{kk'}$ & $q^*(q-2+\bar{\beta}_{k\bull})$ & $q^{*2}$ & $0$ & $q^{*2}$ & $q^*$ \\
$\bar{r}_{k'k}$ & $-$ & $q^*(q-2+\bar{\beta}_{k'\bull})$ & $0$ & $q^{*2}$ & $q^*$ \\
$\bar{r}^0_{mkk'}$ & $-$ & $-$ & $
q^*$ & $-$ & $0$ \\
$\bar{r}^1_{mkk'}$ & $-$ & $-$ & $-$ & $q^*(q-2+\bar{\beta}_{m\bull})$ & $q^*$ \\
$\bar{s}_{kk'}$ & $-$ & $-$ & $-$ & $-$ & $1$
\end{tabular}
\label{tab:nota_varcov2}
\end{table}
\end{proof}

\begin{remark}\label{rem:tfobi}
The limiting distributions of the TFOBI estimates, $\hat{\boldsymbol{\Gamma}}_m := \hat{\textbf{W}}{}_m^\top \hat{\textbf{S}}{}_m^{-1/2}$, $m = 1, \ldots ,r$, follow straightforwardly from the results of the matrix case; using the $m$-flattening of tensors from Section \ref{section:tensor} we can express the $m$-mode tensor product as $\mbf Z \odot \lo {-m} \mbf Z = \mbf Z \loo m \mbf Z \loo m ^\top$, where the matrices $\mbf Z \loo m$, $m = 1, \ldots ,r$, obey the matrix independent component model and have distinct kurtosis row means. Thus the task of finding the $m$th rotation in TFOBI reduces to that of finding the left rotation in MFOBI. Additionally, \eqref{eq:tensor_flatten_kron} shows that the standardization matrices of modes other than $m$ are in the $m$-flattening of the standardized observations collected to the multiple Kronecker product on the right-hand side both satisfying the assumption $\hat{\textbf{R}} \rightarrow_P \textbf{I}$ and contributing nothing to the asymptotics of mode $m$, as shown in the proof of Theorem \ref{theo:micm_asymp}. The limiting distributions for $\hat{\boldsymbol{\Gamma}}_m$ are thus obtained by applying Theorem \ref{theo:micm_asymp} into the empirical distributions of $\mbf Z \loo m$, $m = 1, \ldots ,r$.
\end{remark}

Comparison of the expressions for the two choices of $N$ in Corollary \ref{cor:micm_asymp} immediately yields the following result.

\begin{corollary}\label{cor:micm_acomp}
Assume $q > 1$ and denote by $\textbf{Z}^{kk'}$ the matrix obtained by dropping rows $k$ and $k'$ from $\textbf{Z}$, $k \neq k'$. Then, for $p > 2$, the choice $N = 1$ is asymptotically superior to the choice $N = 0$ in estimating $\hat{\gamma}_{kk'}$ if and only if the average kurtosis of the elements of $\textbf{Z}^{kk'}$ is smaller than $2$, i.e., when
\[\frac{1}{p - 2} \sum_{m \neq k, k'} \bar{\beta}_{m\bull} < 2.\]
If $p = 2$ then the methods are asymptotically equivalent regardless of the distribution of $\textbf{Z}$.
\end{corollary}

According to Corollary \ref{cor:micm_acomp}, to justify the use of the normed version ($N=1$) one would have to assume not only one, but several elements of $\textbf{Z}$ to have kurtosis values below 2. To gain some insight on the strictness of the inequality in Corollary \ref{cor:micm_acomp}, we use the moment inequality of \cite{klaassen2000squared} stating that for unimodal distributions with finite fourth moments we have
\[\gamma^2 \leq \beta - \frac{189}{125}.\]
Combining this bound with Corollary \ref{cor:micm_acomp} then reveals that a necessary condition for the superiority of the normed version is that most elements of $\textbf{Z}$ must be multimodal or almost symmetric (average squared skewness has to be smaller than $0.488$). In the second simulation study of Section \ref{sec:simu} we will conduct a comparison of the two versions under different settings but as the condition in Corollary \ref{cor:micm_acomp} is in general very restrictive and unrealistic the other simulation studies are done using the non-normed versions of MFOBI and TFOBI.

To provide more insight into the second part of Corollary \ref{cor:micm_acomp} where $p = 2$, recall that the Cayley-Hamilton theorem states that every square matrix $\textbf{A} \in \mathbb{R}^{p \times p}$ is annihilated by its characteristic polynomial \citep{roman2005advanced}. For $p=2$ this takes the simple form
\[\textbf{A}^2 - \mathrm{tr}(\textbf{A}) \textbf{A} + \mathrm{det}(\textbf{A}) \textbf{I}_2 = \textbf{0}. \]
Assume now that $\textbf{X}_1, \ldots, \textbf{X}_n$ is a sample of tensors of the same size and that the $m$th mode of $\textbf{X}_1$ has length two. Then, $\textbf{X}_i \odot_{-m} \textbf{X}_i$ is of size $2 \times 2$ for all $i = 1, \ldots , n$ and we have
\[ (\textbf{X}_i \odot_{-m} \textbf{X}_i)^2 = \| \textbf{X}_i \|^2_F (\textbf{X}_i \odot_{-m} \textbf{X}_i) - \mathrm{det}(\textbf{X}_i \odot_{-m} \textbf{X}_i) \textbf{I}_2, \]
where we have utilized the $m$-flattening, $\mathrm{tr}(\textbf{X}_i \odot_{-m} \textbf{X}_i) = \mathrm{tr}(\textbf{X}_{i(m)}\textbf{X}_{i(m)}^\top) = \| \textbf{X}_{i(m)} \|^2_F = \| \textbf{X}_{i} \|^2_F$. Consequently, the sample estimates of $\textbf{B}_{(m)}^0$ and $\textbf{B}_{(m)}^1$ in Definition \ref{definition:m-mode moments} have a difference proportional to the identity matrix, implying that they have the same sets of eigenvectors. Thus for modes of length two the performances of the normed and non-normed version are not only equivalent in the limit, but equivalent for finite samples as well.

\subsection{Comparing the limiting efficiencies of the FOBI and TFOBI estimates}

As the asymptotic variances in Corollaries \ref{cor:vicm_asymp} and \ref{cor:micm_asymp} are rather complicated and each of them relates only to a single element of a single matrix, to compare them as a whole a more concise measure of asymptotic accuracy is desired. For this we first review the \textit{minimum distance index} (MDI) \citep{ilmonen2010new} computed as
\[ \hat{D}_m := D(\hat{\boldsymbol{\Gamma}}_{(m)}, \boldsymbol{\Omega}_m) = \frac{1}{\sqrt{p_m - 1}} \underset{\textbf{C} \in \mathcal{C}^{p_m}}{\mathrm{inf}}\| \textbf{C} \hat{\boldsymbol{\Gamma}}_{(m)} \boldsymbol{\Omega}_m - \textbf{I}_{p_m} \|_F, \]
where $\boldsymbol{\Omega}_m \in \mathbb{R}^{p_m \times p_m}$ is the true $m$-mode mixing matrix and $\hat{\boldsymbol{\Gamma}}_{(m)}$ is the $m$-mode unmixing matrix estimate. The minimum distance index is a measure of how far away the matrix $\hat{\boldsymbol{\Gamma}}_{(m)} \boldsymbol{\Omega}_m$ is from the identity matrix, invariant to order, scales and signs of rows. The index satisfies $0 \leq \hat{D} \leq 1$ with the value $0$ indicating that $\hat{\boldsymbol{\Gamma}}_{(m)} = \boldsymbol{\Omega}^{-1}$ up to permutation, scaling and sign-change of its rows. The index further obeys the limit result $n ( p_m - 1 ) \hat{D}^2_m \rightarrow_d \mathcal{D}_m$ where $\mathcal{D}_m$ is a distribution with the expected value
\begin{align}\label{eq:limit_ev}
E_m := \sum_{k=1}^{p_m-1}\sum_{k'=k+1}^{p_m} \left\{ ASV(\hat{\gamma}^{(m)}_{kk'}) + ASV(\hat{\gamma}^{(m)}_{k'k}) \right\},
\end{align}
where $\hat{\gamma}^{(m)}_{kk'}$ is the $(k, k')$ element of $\hat{\boldsymbol{\Gamma}}_{(m)}$. Consequently $E_m$, the sum of asymptotic variances of the off-diagonal elements of $\hat{\boldsymbol{\Gamma}}_{(m)}$, provides a single-number measure of the asymptotic performance of TFOBI in the $m$th mode.

However, as FOBI produces only a single number $E_1$ and TFOBI one for each mode, $E_1, \ldots , E_r$, we still need to somehow combine the latter to allow comparisons between FOBI and TFOBI. Both the FOBI unmixing estimate $\hat{\boldsymbol{\Gamma}}$ and the Kronecker product $\hat{\boldsymbol{\Gamma}}_{(r)} \otimes \ldots \otimes \hat{\boldsymbol{\Gamma}}_{(1)}$ of the TFOBI unmixing estimates estimate the inverse of the same matrix $\boldsymbol{\Omega} := \boldsymbol{\Omega}_r \otimes \ldots \otimes \boldsymbol{\Omega}_1$ and thus the comparison should be done between them. A link connecting the minimum distance indices of the Kronecker product $\hat{\boldsymbol{\Gamma}}_{(r)} \otimes \ldots \otimes \hat{\boldsymbol{\Gamma}}_{(1)}$ and its component matrices is given next.

\begin{theorem}\label{theo:kronecker_mdi}
Let the sample $\textbf{X}_1, \ldots , \textbf{X}_n  \in \mathbb{R}^{p_1 \times \ldots \times p_r}$ be generated by the tensor-valued independent component model \eqref{eq:tic} with identity mixing, $\boldsymbol{\Omega}_m = \textbf{I}_{p_m}$, $m = 1, \ldots, r$ (in our case also orthogonal mixing suffices, see below). Assume that the unmixing estimates have the limiting normal distributions $\sqrt{n} \mathrm{vec} ( \hat{\boldsymbol{\Gamma}}_{(m)} - \textbf{I}_{p_m} ) \rightarrow_d \mathcal{N}(\textbf{0}, \boldsymbol{\Sigma}_m)$ and denote $p := p_1 \ldots p_r$. Then we have
\[n (p - 1) \hat{D}^2(\hat{\boldsymbol{\Gamma}}_{(r)} \otimes \ldots \otimes \hat{\boldsymbol{\Gamma}}_{(1)}, \textbf{I}_p) = \sum_{m = 1}^r \frac{p}{p_m} n (p_m - 1) \hat{D}^2(\hat{\boldsymbol{\Gamma}}_{(m)}, \textbf{I}_{p_m}) + o_P(1). \]
\end{theorem}

\begin{proof}
By Theorem 1 in \cite{ilmonen2010new} the left-hand side of the claim equals
\begin{align*}
& n \| \mathrm{off}( \hat{\boldsymbol{\Gamma}}_{(r)} \otimes \ldots \otimes \hat{\boldsymbol{\Gamma}}_{(1)} ) \|^2_F + o_P(1) \\
= & n \| \hat{\boldsymbol{\Gamma}}_{(r)} \otimes \ldots \otimes \hat{\boldsymbol{\Gamma}}_{(1)} \|^2_F - n \| \mathrm{diag}( \hat{\boldsymbol{\Gamma}}_{(r)} \otimes \ldots \otimes \hat{\boldsymbol{\Gamma}}_{(1)} ) \|^2_F + o_P(1) \\
= & n \prod_{m=1}^r \| \hat{\boldsymbol{\Gamma}}_{(m)} \|^2_F - n \prod_{m=1}^r \| \mathrm{diag} ( \hat{\boldsymbol{\Gamma}}_{(m)} ) \|^2_F + o_P(1) \\
= & n \prod_{m=1}^r \| \hat{\boldsymbol{\Gamma}}_{(m)} \|^2_F - n \prod_{m=1}^r \left( \| \hat{\boldsymbol{\Gamma}}_{(m)} \|^2_F - \| \mathrm{off} ( \hat{\boldsymbol{\Gamma}}_{(m)} ) \|^2_F \right) + o_P(1).
\end{align*}
Focus next on the second product. We have $n \| \mathrm{off} ( \hat{\boldsymbol{\Gamma}}_{(m)} ) \|^2_F = O_P(1)$, $\| \mathrm{off} ( \hat{\boldsymbol{\Gamma}}_{(m)} ) \|^2_F = o_P(1)$ and $\|  \hat{\boldsymbol{\Gamma}}_{(m)}  \|^2_F = p_m + o_P(1)$, meaning that when the product is opened the terms with more than one $\| \mathrm{off} ( \cdot ) \|^2_F$-term are $o_P(1)$. We are thus left with
\begin{align*}
\sum_{m = 1}^r \left( n \| \mathrm{off} ( \hat{\boldsymbol{\Gamma}}_{(m)} ) \|^2_F \prod_{s \neq m}^r p_s \right) + o_P(1),
\end{align*}
and using Theorem 1 in \cite{ilmonen2010new} in the other direction, $n \| \mathrm{off} ( \hat{\boldsymbol{\Gamma}}_{(m)} ) \|^2_F = n (p_m - 1) \hat{D}^2(\hat{\boldsymbol{\Gamma}}_{(m)}, \textbf{I}_{p_m}) + o_P(1)$ gives the claim.
\end{proof}

\begin{corollary}\label{cor:kronecker_mdi}
Under the assumptions of Theorem \ref{theo:kronecker_mdi} the expected value of the limiting distribution of $n (p - 1) \hat{D}^2(\hat{\boldsymbol{\Gamma}}_{(r)} \otimes \ldots \otimes \hat{\boldsymbol{\Gamma}}_{(1)}, \textbf{I}_p)$ is $\sum_{m = 1}^r (p/p_m) E_m$ where $E_m$ is as in \eqref{eq:limit_ev}.
\end{corollary}

Corollary \ref{cor:kronecker_mdi} implies that the comparison between FOBI and TFOBI should be done by comparing the values of $E^*_1$ and $\sum_{m = 1}^r (p/p_m) E_m$ where $E^*_1$ is the value of \eqref{eq:limit_ev} for FOBI. These values will later be plotted in the simulations where the orthogonal equivariance of TFOBI guarantees that Corollary \ref{cor:kronecker_mdi} holds also when the mixing is orthogonal. Finally, Theorem \ref{theo:kronecker_mdi} also provides insight into the general comparison of two arbitrary (transformed) MDI-values, $n (q_1 - 1) D^2_1$ and $n (q_2 - 1) D^2_2$. If the respective mixing matrices are of the size $q_1 \times q_1$ and $q_2 \times q_2$ then the quantities $n q_2 (q_1 - 1) D^2_1$ and $n q_1 (q_2 - 1) D^2_2$ are on the same ``scale''.


\section{Simulation studies and a real data example}\label{sec:simu}

\subsection{On computational issues}

Before the simulations we compare the assumptions between MFOBI and first vectorizing and then using FOBI, hereafter referred to just as FOBI. The difference clearly lies in Assumptions \ref{assu:vicm_asymp} and \ref{assu:micm_asymp}, which simply state that MFOBI makes much less assumptions on the kurtosis values. For reasonably large square $p \times p$ matrices, vectorizing and using FOBI roughly squares the amount of constraints needed for MFOBI ($2p$ vs. $p^2$). However, one has to bear in mind that the nature of the constraints also changes, MFOBI being concerned with the row and column means of kurtoses and FOBI with the individual values.

Secondly, the most computationally intensive parts in both FOBI and MFOBI are the eigendecompositions, the computational complexity of finding the eigendecomposition of a $p \times p$ matrix being roughly $\mathcal{O}(p^3)$ \citep{pan1998complexity}. Thus assuming again observations of size $p \times p$, MFOBI requires four $\mathcal{O}(p^3)$ operations while FOBI needs two $\mathcal{O}(p^6)$ operations, a considerable difference with large $p$. And thirdly, the numbers of estimable parameters are for MFOBI and FOBI $2p^2-1$ and $p^4$, respectively (assuming again that $p=q$).

All the previous issues become even more serious when comparing TFOBI and FOBI: the number of components in FOBI grows exponentially with the order of the tensor while in general TFOBI just has to perform a few more eigendecompositions of much smaller matrices.

All following computations have been made in \textsf{R} \citep{Rcore}, especially using the packages \texttt{abind} \citep{abind}, \texttt{ICS} \citep{ics}, \texttt{JADE} \citep{miettinen2017blind}, \texttt{MASS} \citep{rpackageMASS} and \texttt{tensor} \citep{tensor}. The implementation of TFOBI and several other tensor methods can be found in the package \texttt{tensorBSS} \citep{RtensorBSS}.

\subsection{Separation performance comparison}

In our first simulation we compared the abilities of MFOBI and FOBI to estimate the unmixing matrix and separate the sources. As our setting we chose samples of independent $3 \times 4$ observations $\textbf{Z}_i$, the 12 components of which, depicted in Table \ref{tab:distr}, were standardized to have zero mean and unit variance. Starting from the top left corner and moving down and right the kurtoses of the components are 1.8, 2.4, 3, 4, 5, 6, 7, 8, 9, 11, 13 and 18. The sample sizes considered were $n = 1000, 2000, 4000, 8000, \ldots , 256000$. Furthermore, we considered three types of double mixings, $\textbf{Z}_i \mapsto \textbf{X}_i = \boldsymbol{\Omega}_1 \textbf{Z}_i \boldsymbol{\Omega}_2^\top$, (i) normal distribution, (ii) uniform distribution and (iii) orthogonal matrices uniform with respect to the Haar measure. In the first two cases appropriate square matrices were created having random elements from $\mathcal{N}(0,1)$ or $\mathcal{U}(-1,1)$ respectively.

\begin{table}[t]
\caption{The distributions of the elements of $\textbf{Z}_i$ in the first simulation. $\mathcal{U}(a, b)$ denotes the continuous uniform distribution from $a$ to $b$, $Tri(a, b, c)$ the triangular distribution from $a$ to $b$ with the apex located at $c$ and $InvGauss(\mu, \lambda)$ the inverse Gaussian distribution with mean $\mu$ and shape $\lambda$.}
\label{tab:distr}
\center
\begin{tabular}{|c|c|c|c|}
\hline
$\mathcal{U}(-\sqrt{3}, \sqrt{3})$ & $t_{10}$ & $\chi^2_3$ & $\chi^2_{1.5}$ \\ \hline
$Tri(-\sqrt{6}, \sqrt{6}, 0)$ & $Gamma(3, \sqrt{3})$ & $Gamma(1.2, \sqrt{1.2})$ & $\chi^2_{1.2}$ \\ \hline
$\mathcal{N}(0, 1)$ & $Laplace(0, 1/\sqrt{2})$ & $Exp(1)$ & $InvGauss(1, 1)$ \\ \hline
\end{tabular}
\end{table}

\begin{figure}[t]
  \center
  \includegraphics[width=\textwidth]{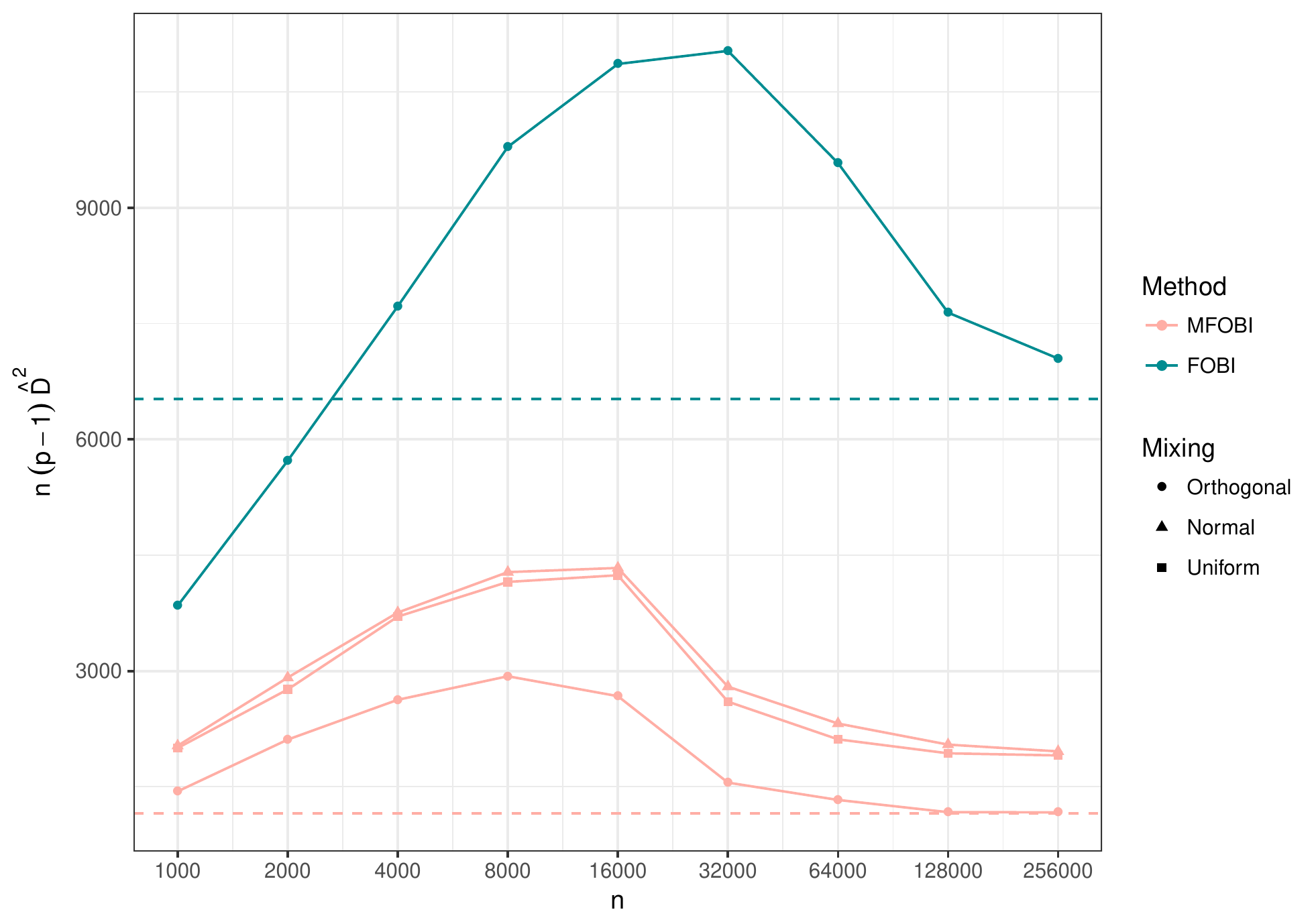}
  \caption{The plot of sample size versus the mean transformed MDI-value with different combinations of method and mixing. The dashed lines give the values of $\sum_{m = 1}^2 (p_1 p_2/p_m) E_m$ and $E^*_1$ towards which the means under orthogonal mixing theoretically converge.}
  \label{fig:mdi}
\end{figure}

We did a total of 2000 replications per setting and as our performance criteria we used the transformed minimum distance indices discussed in the end of Section \ref{sec:asymp}, $n (p_1 p_2 - 1) D(\hat{\boldsymbol{\Gamma}}_{(2)} \otimes \hat{\boldsymbol{\Gamma}}_{(1)}, \boldsymbol{\Omega}_2 \otimes \boldsymbol{\Omega}_1)$ and $n (p_1 p_2 - 1) D(\hat{\boldsymbol{\Gamma}}_{(1)}^*, \boldsymbol{\Omega}_2 \otimes \boldsymbol{\Omega}_1)$, where $\hat{\boldsymbol{\Gamma}}_{(1)}^*$ is the FOBI unmixing estimate. The two values directly measure the accuracies of the methods' separation abilities (lower is better) and under orthogonal mixing (under all mixings for the affine equivariant FOBI), when $n$ grows their means will converge to $\sum_{m = 1}^2 (p_1 p_2/p_m) E_m$ and $E^*_1$, respectively, see \eqref{eq:limit_ev} and Corollary \ref{cor:kronecker_mdi}. The mean values of the criteria and their limit values are plotted in Figure \ref{fig:mdi} and we make the following observations: Contrary to FOBI, the performance of MFOBI indeed depends on the mixing matrix as is shown by the three distinct lines in Figure \ref{fig:mdi}. The separation is easiest for MFOBI when the mixing is orthogonal (because of its orthogonal equivariance orthogonal mixing is equivalent to no mixing at all) and between normal and uniform mixing the separation is slightly easier under the latter. FOBI, while affine equivariant and independent of the choice of mixing, is clearly inferior to MFOBI both with finite samples (the solid lines) and in the limit (the dashed lines). Both curves under orthogonal mixing approach the corresponding limit values, MFOBI faster than FOBI, giving empirical proof on the correctness of the results of Section \ref{sec:asymp}.

\subsection{Comparison between the normed and non-normed versions}

Our next simulation study compares the two choices of TFOBI functionals, $N = 0, 1$. By Corollary \ref{cor:micm_acomp} the value of $N$ makes no difference in  modes of length two and, guided by the condition in Corollary \ref{cor:micm_acomp}, we consider two settings, both random samples of independent and identically distributed $3 \times 3$ matrices, with the elements 
\[
\begin{pmatrix}
\mathcal{N}(0, 1) & \mathcal{B}(-1, 1) & \mathcal{B}(-1, 1) \\
\mathcal{B}(-1, 1) & \mathcal{U}(-\sqrt{3}, \sqrt{3}) & \mathcal{B}(-1, 1) \\
\mathcal{B}(-1, 1) & \mathcal{B}(-1, 1) & \mathcal{B}(-1, 1)
\end{pmatrix}
\quad \mbox{and} \quad
\begin{pmatrix}
\mathtt{\mathcal{B}(-1, 1)} & \mathcal{N}(0, 1) & \mathcal{N}(0, 1) \\
\mathcal{N}(0, 1) & \mathcal{U}(-\sqrt{3}, \sqrt{3}) & \mathcal{N}(0, 1) \\
\mathcal{N}(0, 1) & \mathcal{N}(0, 1) & \mathcal{N}(0, 1)
\end{pmatrix},
\]
where $\mathcal{N}(0, 1)$ is the standardized normal distribution, $\mathcal{U}(-\sqrt{3}, \sqrt{3})$ is the continuous uniform distribution from $-\sqrt{3}$ to $\sqrt{3}$ and $\mathcal{B}(-1, 1)$ is the two-point probability distribution taking equally likely each of the values, $-1$ and $1$. The distributions have the respective kurtoses 3, 1.8 and 1 and consequently the condition of Corollary \ref{cor:micm_acomp} is satisfied for every off-diagonal element in the first setting and is not satisfied for any off-diagonal element in the second setting. Asymptotically the choice $N = 1$ is superior to $N = 0$ in the first setting and vice versa for the second one. To investigate whether this holds also for finite samples we simulated samples of size $n = 1000, 2000, 4000, 8000, \ldots , 256000$ from the above distributions and applied MFOBI to them in four different forms: using the pairs $(\textbf{B}_0^L; \textbf{B}_0^R)$, $(\textbf{B}_1^L; \textbf{B}_1^R)$ and the mixed pairs $(\textbf{B}_0^L; \textbf{B}_1^R)$ or $(\textbf{B}_1^L; \textbf{B}_0^R)$. Intuitively, the performances of the latter two should fall somewhere between those of the former two. To be able to utilize our asymptotic results we did not mix the observations (which is equivalent to using orthogonal mixing).

\begin{figure}[t]
  \center
  \includegraphics[width=\textwidth]{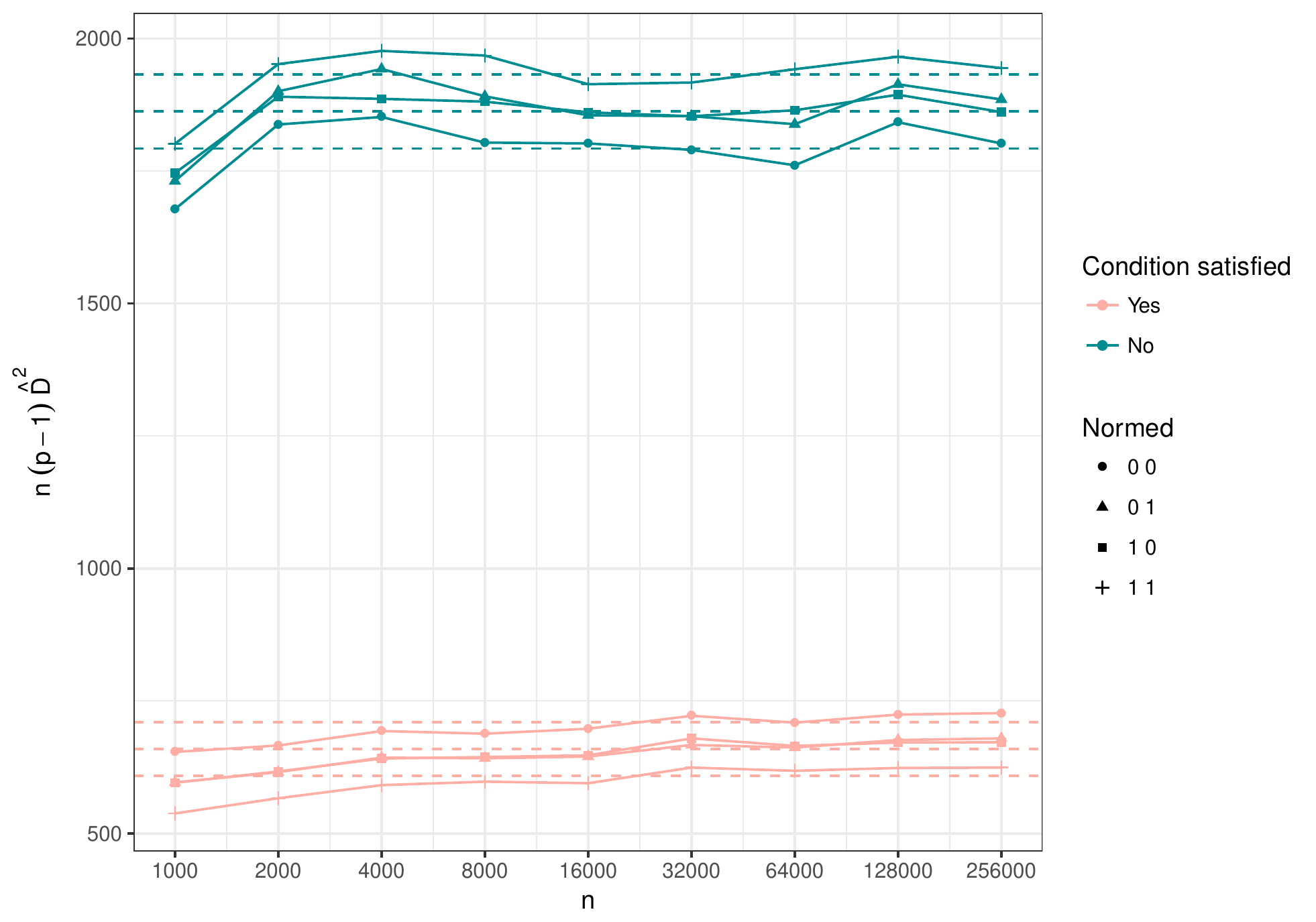}
  \caption{The plot of sample size versus the mean transformed MDI-value with different combinations of setting and $N = 0, 1$. The value of ``Normed'' tells which value of $N$ was used for the left and right unmixing matrices, e.g., 1 0 means that the left unmixing matrix used the normed version but the right one did not. The dashed lines give the values of $\sum_{m = 1}^2 (p_1 p_2/p_m) E_m$ towards which the means theoretically converge.}
  \label{fig:mdi_2}
\end{figure}

We again used the minimum distance index as a criterion and the resulting mean transformed MD-indices over 2000 replications are shown in Figure \ref{fig:mdi_2}.
The dashed lines in the plot indicate the limiting expected values computed using the results of Section \ref{sec:asymp}, toward which the solid lines theoretically converge. We have not visually distinguished the limit lines from each other as their order is the same as the order of the empirical lines. The symmetry of the simulated matrices causes the two mixed MFOBI-functionals to have the same limiting values and similar behavior is also visible in the corresponding two empirical lines matching each other closely. Further observations include: The empirical lines approach the limits rather nicely, with some swaying in the setting where the condition is not satisfied. The setting where the condition is satisfied is overall more easily separated (the lines are lower in the plot). Finally, the ordering between the methods is consistent throughout the study and under both settings the two mixed cases are located almost halfway between the non-mixed cases. Despite the success of the choice $N=1$ here, based on the extreme measures that were required to create a setting where the condition of Corollary \ref{cor:micm_acomp} is satisfied (we needed to resort to the transformed Bernoulli-distribution $\mathcal{B}(-1, 1)$, the probability distribution with the lowest possible kurtosis) we still choose to advocate using primarily the case $N = 0$.

\subsection{FOBI and TFOBI in classification}

Traditionally, although not consistent with the model assumptions, ICA methods are often used as a preprocessing step for classification as linear combinations of the variables with high or low kurtosis are often the most informative in this sense. \cite{PenaPrietoViladomat:2010} for example used FOBI to reveal cluster structures in the data. Also, \cite{tyler2009invariant} showed that two scatter matrices can be combined to estimate Fisher's linear subspace in the case of mixtures of elliptical distributions with proportional covariance matrices. Following the interpretation of FOBI and TFOBI as a combination of different scatters we compare in this section FOBI and TFOBI for the purposes of classification.

The comparison was done in the following set-up. For each replication we simulated 500 observations of $5 \times 5 \times 5$ tensors $\textbf{X}_i$ belonging to one of two groups. In group 1 all elements of the observations $\textbf{X}_i$ are sampled from independent $\mathcal{N}(0, 1)$-distributions, while in group 2 the front upper left $2 \times 2 \times 2$ corner has elements sampled from independent $\mathcal{N}(2, 1)$-distributions (and the rest of the elements from $\mathcal{N}(0, 1)$). A proportion $\pi$ of all observations belonged to group 2.

We did 2000 replications for each of the values $\pi = 0.10, 0.15,  \ldots , 0.50$ and for each replication we mixed the observations from all three $m$-modes 
using the same three types of mixing matrices as in the previous section.
Next, we divided the transformed data randomly into training and test sets, with the respective sizes of $400$ and $100$. Both TFOBI and FOBI were then carried out for the training data and linear discriminant analysis (LDA) was used to create classification rules based on certain selected components. For TFOBI we chose these to be the corner components $z_{1,1,1}$ and $z_{5,5,5}$ and the components having the highest and lowest kurtoses. For FOBI we simply chose the first two and the last two components (ordered according to kurtosis). As a reference, we also created a classification rule with LDA using all the original components. The means of the proportions of correct predictions in the test set for each of the rules are plotted in Figure \ref{fig:qda}. The reference value is included as the line ``NONE''.

\begin{figure}[t]
  \center
  \includegraphics[width=\textwidth]{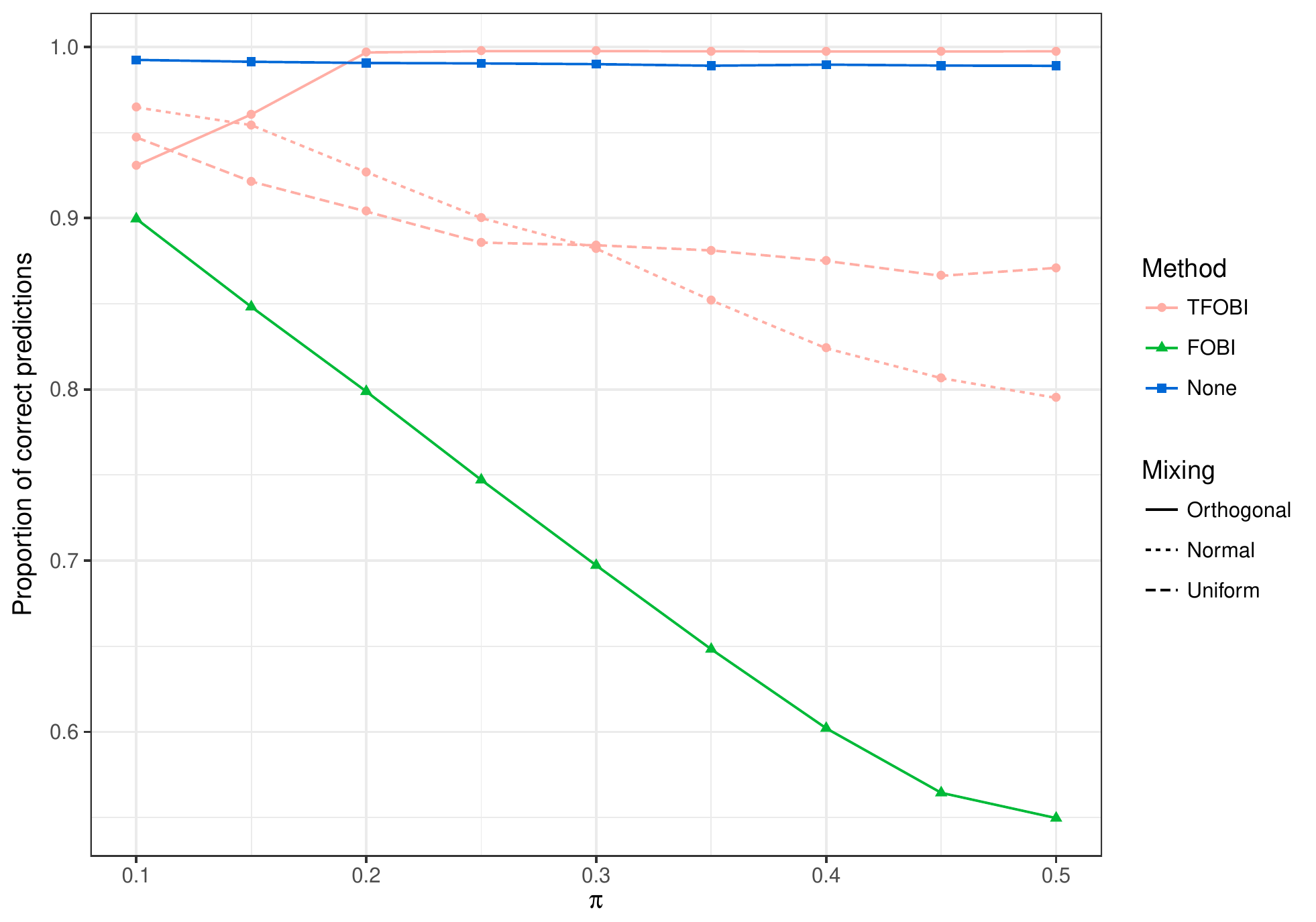}
  \caption{Proportions of correct classifications as a function of proportion $\pi$ with FOBI and TFOBI as pre-steps and with three types of random mixing matrices.}
  \label{fig:qda}
\end{figure}

LDA uses the training set group proportions as a prior and a ``baseline'' proportion of correct predictions is thus $1-\pi$, corresponding to classifying all test observations to the dominant group. The plot indicates that FOBI cannot find the direction separating the groups in any systematic way and is actually no better than the baseline. TFOBI, on the other hand, is in every case better than FOBI and performs very nicely under all mixings (especially orthogonal). Under orthogonal mixing and for $\pi$ larger than or equal to 0.20 TFOBI, being able to filter out the noise, is also slightly better than using all the original components. The simulation thus implies that TFOBI provides a reliable way of extracting the separating variables from tensor-valued data.


\subsection{Real data example}

To see how MFOBI works with real data we use the \textit{semeion}\footnote{Semeion Research Center of Sciences of Communication, via Sersale 117, 00128 Rome, Italy; Tattile Via Gaetano Donizetti, 1-3-5,25030 Mairano (Brescia), Italy.} data set available from the UCI Machine Learning Repository \citep{Lichman:2013}. The data consist of 1593 scanned handwritten digits written by 80 persons represented as binary $16 \times 16$ matrices. For our analysis we picked only the images of the visually similar digits $3$ and $8$ hoping to find a direction separating the two digits. The number of observations is then $n=314$ with almost equal number of threes and eights (159 and 155, respectively).

\begin{figure}[t]
  \center
  \includegraphics[width=\textwidth]{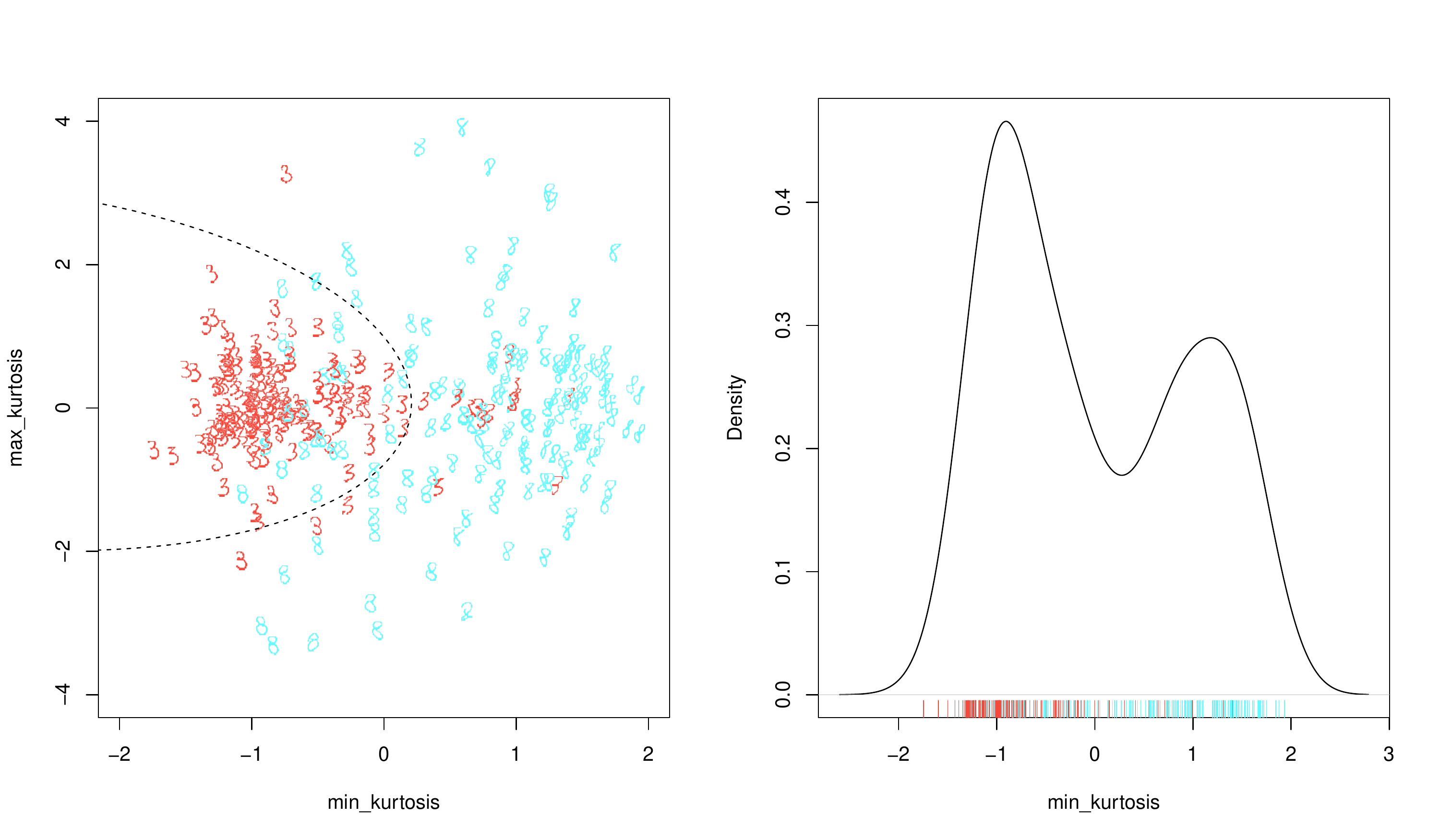}
  \caption{The figure on the left-hand side shows the scatter plot of the two independent components having the lowest and highest kurtoses, dividing the data nicely into two groups. The separation is also visible on the right-hand side in the rug and the bimodal kernel density estimate of the component with the lowest kurtosis.}
  \label{fig:semeion_comp}
\end{figure}

The results of MFOBI are shown in Figure \ref{fig:semeion_comp}. The scatter plot on the left shows the distributions of the components having the highest and lowest kurtoses ($z_{1,2}$ and $z_{16,16}$, respectively), with the individual images as plotting markers, along with the decision boundary given by quadratic discriminant analysis. Although the two groups of digits overlap a bit the separation is still very clear, as is evidenced also by the kernel density estimate of the minimal kurtosis component on the right-hand side of Figure \ref{fig:semeion_comp}. We also see that the hand-writing is slanting more and more to the right with increasing values of $z_{16,16}$ and that the variable $z_{1,2}$ with highest kurtosis can be used in search for outliers. 


For comparison, we also tried applying regular FOBI to the vectorized data with somewhat disappointing results; the covariance matrix of the full data was not invertible and when trying with some subsets of the data, FOBI succeeded only in finding a few outliers.

\section{Concluding remarks}\label{sec:conc}

In this paper, we presented methods of independent component analysis for matrix- and tensor-valued observations called MFOBI and TFOBI. The total procedure can be seen as a simultaneous application of the classic FOBI on all $m$-modes of the observed tensors. 

Apart from the algorithms and two different ways of estimating the unmixing matrix we also provided the asymptotic variances of the elements of the unmixing matrix estimates in the case of orthogonal mixing. The variance expressions then show that using the non-normed version of TFOBI is in most cases the preferable approach. Regarding the comparison of TFOBI with the often used combination of vectorizing and FOBI, we first stated that the numbers of estimable parameters and assumptions required are of much smaller order in MFOBI and TFOBI. This is because they are able to exploit the possible tensor structure in the estimation. Next, simulations were used to show TFOBI's superiority to FOBI also in practice, both in estimating the unmixing matrix and as a preprocessing step for discriminant analysis.

With MFOBI and TFOBI being derivatives of FOBI a reasonable conjecture is that, instead of relying on the kurtosis matrices $\textbf{B}{}^N$, extending some other standard ICA techniques like projection pursuit or JADE \cite{cardoso1993blind} into the tensor case would lead into better estimates. \cite{virta2016jade} showed that this holds for JADE and some preliminary investigation shows that this is indeed the case for projection pursuit as well and such a take on the problem can then be seen as a tensor version of FastICA \citep{HyvarinenKarhunenOja:2001}. The resulting concept of tensorial projection pursuit will be addressed in future work.

Nevertheless, compared with other perhaps more sophisticated routes of generalization, the FOBI-type extensions enjoy a particularly simple structure for high-dimensional tensors: the higher moment tensors decompose neatly to matrices of reasonably low dimensions. As a result the eigendecompositions only need to be performed on $p_m \times p_m$ matrices individually. This feature makes MFOBI and TFOBI especially attractive when applied on a large scale.

\section*{References}

\bibliography{references}

\end{document}